\documentclass[12pt,twoside]{amsart}
\usepackage{amssymb,amsmath,amsthm}
\usepackage{verbatim}
\usepackage{graphicx}
\usepackage{epsfig}
\usepackage{color, soul}
\usepackage[all]{xy}

\textheight
23.5cm \textwidth 16.5cm

\DeclareMathAlphabet{\mathpzc}{OT1}{pzc}{m}{it}
\newcommand{\N}{{\ensuremath{\mathbb{N}}}}

\newtheorem{theorem}{Theorem}[section]
\newtheorem{lemma}[theorem]{Lemma}
\newtheorem{definition}[theorem]{Definition}
\newtheorem{corollary}[theorem]{Corollary}
\newtheorem{proposition}[theorem]{Proposition}
\newtheorem{remark}[theorem]{Remark}
\newtheorem{example}[theorem]{Example}

\def\d{\displaystyle}

\def\pco{\text{{\it p}{\rm -co}}}
\def\co{{\rm co}}
\def\ep{\varepsilon}
\def\m{\mathpzc{m}}
\def\ker{{\rm ker\, }}
\def\dist{{\rm dist}}

\title[]{On $p$-Compact mappings and $p$-approximation}

\author{Silvia Lassalle and Pablo Turco}

\thanks{This project was supported in part by UBACyT X218 and UBACyT X038}
 \address{Departamento de Matem\'{a}tica - Pab I,
 Facultad de Cs. Exactas y Naturales, Universidad de Buenos
Aires, (1428) Buenos Aires, Argentina}

\email{slassall@dm.uba.ar, paturco@dm.uba.ar}

 \keywords{$p$-compact sets, holomorphic mappings, approximation properties} 

\begin{document}
\parindent 10pt
\parskip .2cm

\begin{abstract} The notion of $p$-compact sets arises naturally from Grothendieck's characterization of compact sets as those contained in the convex hull of a norm null sequence. The definition, due to Sinha and Karn (2002), leads to the concepts of $p$-approximation property and $p$-compact operators, which form a ideal with its ideal norm $\kappa_p$. This paper examines the interaction between the $p$-approximation property and certain space of holomorphic functions, the $p$-compact analytic functions.  In order to understand these functions we  define a $p$-compact radius of convergence which allow us to  give a characterization of the functions in the class. We show that 
$p$-compact holomorphic functions behave more like nuclear than compact maps.  We use the $\epsilon$-product  of Schwartz, to  characterize the $p$-approximation property of a Banach space in terms of $p$-compact homogeneous polynomials and  in terms of $p$-compact holomorphic functions with range on the space. Finally, we show that $p$-compact holomorphic functions  fit into the framework of holomorphy types which allows us to inspect the $\kappa_p$-approximation property. Our approach also allows us to solve several questions posed by Aron, Maestre and Rueda in \cite{AMR}.
\end{abstract} 
\maketitle

\section*{Introduction}

In the Theory of Banach spaces (or more precisely, of infinite dimensional
locally convex spaces),  three concepts appear systematically
related since the foundational articles by Grothendieck \cite{Gro} and
Schwartz \cite{Schw}. We are referring to  compact sets, compact operators and
the approximation property. A Banach space $E$ has the approximation property whenever the identity
map can be uniformly approximated by finite rank operators on compact sets. Equivalently, if
$E'\otimes E$, the subspace of finite rank operators, is dense in $\mathcal L_c(E;E)$, the space of
continuous linear operators considered with the uniform convergence on compact sets. The other
classical reformulation states that $E$ has the approximation property if $F'\otimes E$  is dense in $\mathcal K(F;E)$,  the space of compact operators, for all Banach spaces $F$.  It was not until 1972 that Enflo
provided us with the first example of a Banach space without the approximation property \cite{En}.
In the quest to a better understanding of this property and the delicate relationships inherit on
different spaces of functions, important variants of the approximation property have
emerged and were intensively studied. For the main developments on the subject we quote  the
comprehensive survey \cite{Cas} and the references therein. 

Inspired in Grothendieck's result which characterize relatively compact sets as those
contained in the convex hull of a norm null sequence of vectors of the space, Sinha and Karn \cite{SiKa}
introduced the concept of relatively $p$-compact sets. Loosely speaking, these sets are determined
by norm  $p$-summable sequences. Associated to relatively $p$-compact sets we have naturally defined
the notions of  $p$-compact operators and the $p$-approximation property (see definitions
below). Since relatively $p$-compact  sets are, in particular, relatively compact, the
$p$-approximation property can be seen as a way to weaken the approximation property. These three concepts were first studied  by Sinha and Karn \cite{SiKa, SiKa2} and,
more recently, several authors continued the research on this subject  \cite{CK, DOPS, DPS_dens, DPS_adj, GaLaTur}. 

This paper examines the interaction between the $p$-approximation property and the class of
$p$-compact holomorphic functions. The connection between  the approximation property 
and  the space of holomorphic functions is not without precedent. The pioneer article on this topic is due to 
Aron and Schottenloher \cite{AS}, who prove  that a Banach space $E$ has the approximation property if and only if $(\mathcal H(E), \tau_0)$, the space of the entire functions with the compact open topology, has the approximation property.  Since then, many authors studied the approximation property for spaces of holomorphic functions in different contexts,  see for instance \cite{BDR, Cal, DiMu1, DiMu2, Mu2}. Recently, Aron, Maestre and Rueda
\cite{AMR} prove that $E$ has the $p$-approximation property if and only if
$(\mathcal H(E), \tau_{0p})$ has the approximation property, here $\tau_{0p}$ denotes the
topology of the uniform  convergence  on $p$-compact sets.  The relation between the
approximation property and holomorphic mappings was studied in detail in
\cite{AS},  where the class of compact holomorphic functions play a crucial
role. 

The article is organized as follows. In the first section we fix the notation and state some basic results on $p$-compact mappings. In Section~2 we study the behavior of $p$-compact homogeneous polynomials which can be considered as a polynomial Banach ideal with a natural norm denoted by $\kappa_p$. Following \cite{CK} we show that any $p$-compact homogeneous polynomial factors through a quotient of $\ell_1$ and express the $\kappa_p$-norm in terms of an infimum of certain norms of all such possible factorizations. This result slightly improves \cite[Theorem 3.1.]{CK}. We also show that the Aron-Berner extension preserves the class of $p$-compact polynomials with the same norm. Finally we show that there is an isometric relationship between
 the adjoint of $p$-compact polynomials and quasi $p$-nuclear operators, improving the analogous result for operators  \cite{DPS_adj}. 

Section~3 is devoted to the study of $p$-compact holomorphic mappings. Since $p$-compact functions are  compact, we pay special attention to the results obtained by Aron and Schottenloher \cite{AS}, where the authors prove that any holomorphic function is compact  if and only if each polynomial of its Taylor series expansion at $0$ is compact, \cite[Proposition 3.4]{AS}. Then, Aron, Maestre and Rueda \cite[Proposition 3.5]{AMR} show that each component of the Taylor series expansion of a $p$-compact holomorphic mapping has to be $p$-compact. We  define a natural $p$-compact radius of convergence and, in Proposition~\ref{p-comp radius}, we give a characterization of this type of functions. Surprisingly, we found that $p$-compact holomorphic functions behave more like nuclear than compact mappings.  We show this feature with  two examples.  Example~\ref{Pn p-comp f no} shows that  Proposition~\ref{p-comp radius} cannot be improved and also that it is possible to find an entire function whose polynomials at $0$ are $p$-compact but the function fails to be $p$-compact at $0$, which answers by the negative the question posed in \cite[Problem 5.2]{AMR}. In Example~\ref{f p-com en 0 no p-comp} we construct an entire function on $\ell_1$  which is $p$-compact on the open unit ball, but it fails to be $p$-compact at the first element of the canonical basis of $\ell_1$, giving an answer to \cite[Problem 5.1]{AMR}. 

We apply the results of Section~2 and~3 to study the $p$-approximation property in Section~4. We characterize the $p$-approximation property of a Banach space in terms of $p$-compact homogeneous polynomials with range on the space. Our proof requires the notion of the $\epsilon$-product of Schwartz \cite{Schw}. We also show that a Banach space $E$ has the $p$-approximation property if and only if $p$-compact homogeneous polynomials with range on $E$ can be uniformly approximated by finite rank polynomials. Then, we give the analogous result for $p$-compact holomorphic mappings endowed with the Nachbin topology, Proposition~\ref{Nach-p-compact}. 

The final section is dedicated to the $p$-compact holomorphic mappings withing the framework of holomorphy types, concept introduced by Nachbin \cite{NACH, NACH2} .  This allows us to inspect the $\kappa_p$-approximation property introduced, in \cite{DPS_dens}, in the spirit of \cite[Theorem 4.1]{AS}.

For  general background on the approximation property and its different variants we refer the reader to  \cite{Cas, LT}.

\section{Preliminaries}

Throughout this paper $E$ and $F$ will be Banach spaces. We denote  by 
$B_E$ the closed unit ball of $E$,  by $E'$ its topological dual, and by $\ell_p(E)$ the Banach space of the $p$-summable sequences of elements of $E$, endowed with the natural norm.  Also, $c_0(E)$ denotes the space of null sequences of $E$ endowed with the supremum norm. Following \cite{SiKa}, we say that a subset $K\subset E$ is relatively $p$-compact, $1\le p\le \infty$,  if there exists a sequence $(x_n)_n\subset \ell_p(E)$ so that $K$ is
contained in the closure of $\{\sum \alpha_n x_n\colon (\alpha_n)_n\in
B_{\ell_{q}}\}$, where $B_{\ell_{q}}$ denotes the closed unit ball of
$\ell_{q}$, with $\frac 1p + \frac 1{q} =1$. We denote this set by $\pco\{x_n\}$
and its closure by $\overline{\pco}\{x_n\}$.  Compact sets are considered
$\infty$-compact sets, which corresponds to $q=1$. When $p=1$, the $1$-convex
hull is obtained by considering coefficients in $B_{\ell_\infty}$ or, if
necessary, with some extra work by coefficients in  $B_{c_0}$, see
\cite[Remark 3.3]{DPS_adj}.

Since the sequence $(x_n)_n$ in the definition of a relatively $p$-compact set $K$ converges to zero, any
$p$-compact set is compact. Such a sequence is not unique, then we may consider 
$$
\m_p(K; E)=\inf\{\|(x_n)_n\|_p\colon K\subset \pco\{x_n\}\}
$$
which measures the size of $K$ as a $p$-compact set of $E$. For simplicity,
along this work we write $\m_p(K)$ instead of $\m_p(K; E)$. When  $K$ is  relatively 
$p$-compact  and $(x_n)_n\subset \ell_p(E)$  is a sequence  so that $K\subset \pco\{x_n\}$, any $x\in K$ has the form  $x=\sum \alpha_n x_n$ with $(\alpha_n)_n$ some sequence in $B_{\ell_{q}}$. 
By H\"older's inequality, we have that $\|x\|\le \|(x_n)_n\|_{\ell_p(E)}$ and in consequence,
$\|x\| \le \m_p(K)$, for all $x\in K$. We will  use without any further mention the following  equalities:  $\m_p(K) = \m_p(\overline K)=\m_p(\Gamma(K))$, where  $\Gamma(K)$ denotes 
the absolutely convex hull of $K$, a  relatively $p$-compact set.

The space of linear bounded operators from $E$ to $F$ will be denoted by $\mathcal L(E;F)$ and $E'\otimes F$ will denote its subspace of finite rank operators.  As in  \cite{SiKa},  we say that
an operator $T\in \mathcal L(E;F)$ is  $p$-compact,  $1\le p\le \infty$,  if $T(B_E)$ is a relatively $p$-compact set in $F$. The space of $p$-compact operators from $E$ to $F$ will be denoted by $K_p(E;F)$. If $T$ belongs to  $K_p(E;F)$, we define 
$$
\kappa_p(T)=\inf\big\{ \|(y_n)_n\|_p\colon (y_n)_n\in
\ell_p(F) \ \text{and} \ T(B_E)\subset \pco\{y_n\}\big\},
$$ 
where $\kappa_\infty$ coincides with the supremum norm. 
It is easy to see that $\kappa_p$ is a norm on $K_p(E;F)$ and following \cite{PerPi} (see also \cite{DPS_adj}) it is possible to show that the pair $(K_p, \kappa_p)$ is a Banach operator ideal. 

The Banach ideal of $p$-compact operators is associated by duality with the ideal of quasi-$p$-nuclear operators, introduced and studied by Persson and Pietsch \cite{PerPi}.  
A linear operator $T\in \mathcal L(E;F)$  is
said to be quasi $p$-nuclear if 
$j_F T\colon E\to  \ell_\infty(B_{F'})$ is a $p$-nuclear operator, where $j_F$ is the
natural isometric embedding from $F$ into $\ell_\infty(B_{F'})$. It is known that an operator $T$ is quasi $p$-nuclear if and only if there exists a sequence $(x'_n)_n\subset \ell_p(E')$, such that 
$$
\|Tx\| \le \Big(\sum_n |x'_n(x)|^p\Big)^{\frac 1p},
$$
for all $x\in E$ and the quasi $p$-nuclear norm of $T$ is given by 
$$
\nu_p^Q(T)=\inf\{\|(x'_n)_n\|_p\colon \|Tx\| ^p\le \sum_n |x'_n(x)|^p, \quad \forall x\in E\}.
$$ 
The space of quasi $p$-nuclear operators from $E$ to $F$ will be denoted by $\mathcal {QN}_p(E;F)$.
The duality relationship is as follows. Given $T\in \mathcal L(E;F)$, $T$ is  $p$-compact  if and only if its adjoint is quasi $p$-nuclear. Also, $T$ is  is quasi $p$-nuclear if and only if its adjoint is $p$-compact, see \cite[Corollary~3.4 ]{DPS_adj} and \cite[Proposition~3.8]{DPS_adj}.

A mapping $P\colon E\to F$ is an $m$-homogeneous polynomial if there exists a
(unique) symmetric $m$-linear form $\overset\vee
P\colon\underbrace{E\times\cdots\times E}_m\to
F$ such that
$$
P(x)=\overset\vee P(x,\dots,x),
$$
for all $x\in E$. The space of $m$-homogeneous continuous polynomials from $E$
to $F$ will be denoted by $\mathcal{P}(^mE; F)$, which  is a Banach space considered with the supremum
norm
$$
\|P\|=\sup\{\|P(x)\|\colon \, x\in B_E\}.
$$

Every $P\in\mathcal{P}(^mE,F)$ has associated two natural
mappings: the {\it linearization} denoted by 
$L_P\in \mathcal{L}(\bigotimes^m_{\pi_s}E ; F)$, where $\bigotimes^m_{\pi_s}E$ stands for the completion of the symmetric $m$-tensor product endowed with the symmetric projective  norm. Also we have
the polynomial $\overline{P}\in \mathcal P(^m E'', F'')$, which is the
canonical extension of $P$ from $E$ to $E''$ obtained by
weak-star density, known as the Aron-Berner extension of
$P$ \cite{AB}. While $\|L_P\|\le \frac{m^m}{m!}\|P\|$, $\overline P$ is an isometric extension of $P$ \cite{DG}.

A mapping $f\colon E\to F$ is holomorphic at $x_0\in E$ if there exists a sequence of polynomials $P_mf(x_0) \in \mathcal{P}(^mE,F)$ such that  
$$
f(x)=\sum_{m=0}^\infty P_mf(x_0)(x-x_0),
$$ 
uniformly for all $x$ in some neighborhood of $x_0$. We say that  $
\sum_{m=0}^\infty P_mf(x_0),
$
is the Taylor series expansion of $f$ at $x_0$ and that $P_mf(x_0)$ is the $m$-component of the series at $x_0$. A function is said to be holomorphic or entire if it is holomorphic at $x$ for all $x\in E$. The space of entire functions from $E$ to $F$ will be denote by  $\mathcal H(E;F)$.  

We refer the reader to \cite{DIN2, Mu} for general background on polynomials and holomorphic functions.

\section{The $p$-compact polynomials}

We want to understand the behavior of of $p$-compact
holomorphic mappings. The definition, due to Aron, Maestre
and Rueda \cite{AMR} was introduced as a natural extension
of $p$-compact operators to the non linear case. In
\cite{AMR} the authors show that for any $p$-compact
holomorphic function each $m$-homogeneous polynomial of its
Taylor series expansion must be $p$-compact. Motivated by
this fact we devote this section to the study of
polynomials. 

Recall that  $P\in {\mathcal P}(^mE;F)$ is said to be
$p$-compact, $1\leq p \leq \infty$, if there exists a
sequence $(y_n)_{n} \in \ell_p(F)$, $(y_n)_{n} \in c_0$ if $p=\infty$,  such that
$P(B_E)\subset \{\sum_{n=1}^{\infty}\alpha_n y_n \colon
(\alpha_n)_{n} \in B_{\ell_{q}}\}$, where
$\frac 1p + \frac 1{q}=1$. In particular, any $p$-compact
polynomial is compact.
We denote by $\mathcal P_{K_p}(^m E; F)$ the space of
$p$-compact $m$-homogeneous polynomials and by $\mathcal P_{K}(^m E; F)$ the space of compact polynomials.
Following \cite{DPS_dens}, for $P\in \mathcal P_{K_p}(^m E;
F)$ we may define 
$$
\kappa_p(P)=\inf\big\{ \|(y_n)_n\|_p\colon (y_n)_n\in
\ell_p(F) \ and \ P(B_E)\subset \pco\{y_n\}\big\}.
$$ 
In other words,  $\kappa_p(P)=\m_p(P(B_E))$. It is easy
to see that $\kappa_p$ is, in fact, a norm satisfying that
$\|P\|\le \kappa_p(P)$, for any $p$-compact homogeneous
polynomial $P$, and that $(\mathcal P_{K_p}(^m E; F),
\kappa_p)$ is a polynomial Banach ideal. Furthermore, we
will see that any $p$-compact $m$-homogeneous polynomial
factors through a $p$-compact operator and a continuous
$m$-homogeneous polynomial. Also, the $\kappa_p$-norm
satisfies the natural infimum condition.

\begin{lemma}\label{pcomp-lin} 
Let $E$ and $F$ be Banach spaces and let $P \in \mathcal
P(^m E;F)$.  The following statements are  equivalent.

\begin{enumerate}
 \item[(i)] $P$ is $p$-compact.
\item[(ii)] $L_P\colon \bigotimes^m_{\pi_s}E\to F$, the
linearization of $P$, is a $p$-compact operator.
\end{enumerate}
Moreover, we have  $\kappa_p(P)=\kappa_p(L_P)$.
\end{lemma}

\begin{proof} To show the  equivalence, we appeal to the
familiar diagram, where $\Lambda$ is a norm one homogeneous
polynomial ($\Lambda(x)=x^{m}$) and $P=L_P \Lambda$:  
$$
\xymatrix{
& E \ar[r]^P \ar[d]_{\Lambda}  & F   
                 &    \\
& \bigotimes^m_{\pi_s}E \ar[ur]^{L_P}  & }
$$

Note that  the open unit ball of $\bigotimes^m_{\pi_s}E$ is
the absolutely convex hull $\Gamma \{x^m\colon \|x\|<1\}$.
Then, we have that
$P(B_E)\subset \Gamma(\{L_P(x^m)\colon \|x\|<1\}) =
\Gamma(P(B_E))$. Now, the equality $L_P(\Gamma \{x^m\colon
\|x\|<1\})=\Gamma(P(B_E))$  shows that  any sequence
$(y_n)_n\in \ell_p(F)$ involved in the definition of
$\kappa_p(P)$ is also involved in the definition of
$\kappa_p(L_P)$ and vice versa, which completes the proof.
\end{proof}

Sinha and Karn \cite[Theorem 3.2]{SiKa} show that a
continuous operator is $p$-compact if and only if it
factorizes via a quotient of $\ell_q$ and some sequence
$y=(y_n)_n \in \ell_p(F)$. Their construction is as follows.
Given $y=(y_n)_n \in \ell_p(F)$ there is a canonical
continuous linear operator $\theta_y\colon \ell_q\to F$
associated to $y$ such that on the unit basis of
$\ell_q$ satisfies $\theta_y(e_n)=y_n$. Then, $\theta_y$
is extended by continuity and density. Associated to
$\theta_y$ there is a natural injective operator
$\widetilde{\theta_y}\colon {\ell_{q}}/{\ker \theta_y}
\rightarrow F$ such that
$\widetilde{\theta_y}[(\alpha_n)_n]=\theta_y((\alpha_n)_n)$
with $\|\widetilde
\theta_y\|=\|\theta_y\|\le\|y\|_{\ell_p(F)}$.
Now, if we start with $T$ belonging to $K_p(E;F)$, there
exists $y=(y_n)_n \in \ell_p(F)$ so that $T(B_E)\subset
\pco\{y_n\}$ and, via $\theta_y$, it is possible to define
an operator $T_y\colon E \rightarrow {\ell_{q}}/{\ker
\theta_y}$ by
$T_y(x)=[(\alpha_n)_n]$ where $(\alpha_n)_n\in \ell_q$ is a
sequence satisfying that $T(x)=\sum_{n=1}^{\infty} \alpha_n
y_n$, which exists by the $p$-compactness of $T$. Note that
the operator $T_y$ is well defined, $\|T_y\|\le 1$ and $T$
satisfies the factorization $T=\widetilde{\theta_y}T_y$,
where $\widetilde{\theta_y}$ is $p$-compact and
$\kappa_p(\widetilde \theta_y)=\|y\|_{\ell_p(F)}$. 
It is clear that, if an operator $S$ admits such factorization, then $S$ is $p$-compact. 

Given a sequence $y=(y_n)_n \in \ell_p(F)$, Choi and Kim 
\cite[Theorem 3.1]{CK} factorize the operator $\theta_y$
through $\ell_1$ as the composition of a compact and a
$p$-compact operator, concluding that any $p$-compact
operator $T$ factors through a quotient space of $\ell_1$ as
a composition of a compact mapping with a $p$-compact
operator.  

The behavior of  $p$-compact polynomials is similar to that
described for $p$-compact operators. In the proposition below
 a slight improvement of \cite[Theorem 3.1]{CK} is obtained.
We prove that the corresponding factorizations via a quotient of $\ell_1$
suffice to characterize $\kappa_p(P)$ for $P$ a $p$-compact
polynomial and therefore, for a $p$-compact operator. In
order to do so, we will use the following technical lemma.
Although we believe it should be a known basic result, we
have not found it  mentioned in the literature as we  need it. Thus,  we
also include a proof. 

First, we fix some notation: for $\sigma \subset \N$ a
finite set, we write $S_{\sigma}(r)=\sum_{n \in \sigma}
r^n$. The following fact has a direct proof. Let $\sigma_1,
\sigma_2,\ldots$ be a disjoint sequence of finite sets such
that its union, $\cup_n \sigma_n$, is infinite.  Take  
$0<r<t<1$ and consider the sequence  $\beta= (\beta_n)_n$
defined  by
$\beta_n=\frac{S_{\sigma_n}(r)}{S_{\sigma_n}(t)}$, for all
$n$. Then, $\beta$ belongs to $B_{c_0}$.

\begin{lemma}\label{lemma-tech}
Let $1\le p < \infty$. Given  $(x_n)_n \in \ell_p$ and
$\ep>0$, there exists $\beta=(\beta_n)_n \in B_{c_0}$ such
that $(\frac{x_n}{\beta_n})_n \in \ell_p$ and
$\|(\frac{x_n}{\beta_n})_n\|_p \leq \|(x_n)_n\|_p(1+\ep)$.
\end{lemma}

\begin{proof}
It is enough to prove the lemma for $(x_n)_n \in \ell_1$
with $\|(x_n)_n\|_1=1$. Indeed, suppose  the result holds
for this case. Fix $(x_n)_n \in \ell_p$,   a nonzero sequence and consider the
sequence $(z_n)_n$ defined by
$z_n=\frac{x_n^p}{\|(x_n)_n\|_p^p}$, which is a norm one
element of  $\ell_1$. Given $\ep>0$,  take $(\alpha_n)_n \in
B_{c_0}$ such that $\|(\frac {z_n}{\alpha_n})_n\|_1 \leq
1+\ep$ and the conclusion follows with
$\beta=(\alpha_n^{1/p})_n$.

Now, suppose $(x_n)_n \in \ell_1$ and $\|(x_n)_n\|_1=1$, we
also may assume that $x_n \neq 0$ for all $n$. Choose
$\delta>0$ such that $\frac{1+\delta}{1-\delta}<1+\ep$. We 
construct $\beta$ inductively.

Since $|x_1|<1$, there exists $m_1 \in \N$ such that if
$\sigma_1=\{1,2,\ldots, m_1\}$, then
$|x_1|<S_{\sigma_1}(\frac{1}{2})$. Let $n_1 \ge 2$ be the
integer so that $\sum_{n
<n_1}|x_n|<S_{\sigma_1}(\frac{1}{2})$ and
$\sum_{n=1}^{n_1}|x_n|\geq S_{\sigma_1}(\frac{1}{2})$. Then,
there exists $0<t_{n_1}\leq 1$ such that

$$
\sum_{n<n_1}|x_n|+t_{n_1}|x_{n_1}|=S_{\sigma_1}({\textstyle
\frac{1}{2}}) \quad \mbox{ and}\quad 
(1-t_{n_1})|x_{n_1}| +
\sum_{n>n_1}|x_n|=\sum_{n>m_1}({\textstyle \frac{1}{2}})^n.
$$
Now, since $(1-t_{n_1})|x_{n_1}| + |x_{n_1+1}|
<\sum_{n>m_1}(\frac{1}{2})^n$, there exists $m_2>m_1$ such
that if  $\sigma_2=\{m_1+1,\ldots, m_2\}$ we have that
$(1-t_{n_1})|x_{n_1}| + |x_{n_1+1}| <
S_{\sigma_2}(\frac{1}{2})$. Let $n_2>n_1+1$ be the integer
such that $(1-t_{n_1})|x_{n_1}| +
\sum_{{n_1}+1}^{n_2-1}|x_n|<S_{\sigma_2}(\frac{1}{2})$ and
$(1-t_{n_1})|x_{n_1}| + \sum_{{n_1}+1}^{n_2}|x_n|\geq
S_{\sigma_2}(\frac{1}{2})$. Then, there exists $0<t_{n_2}\leq
1$ such that 
$$
(1-t_{n_1})|x_{n_1}| + \sum_{n_1 < n < n_2}|x_n| \ + \
t_{n_2}|x_{n_2}|=S_{\sigma_2}({\textstyle \frac{1}{2}})\quad
\mbox{ and}\quad (1-t_{n_2})|x_{n_2}| + \sum_{n>n_2}
|x_n|=\sum_{n> m_2} ({\textstyle \frac{1}{2}})^n.
$$ 
Continuing this procedure, we can find $n_j, m_j \in \N$,
$n_j +1<n_{j+1}, m_j<m_{j+1} \ \forall j$, and
$0<t_{n_j}\leq 1$ such that, if $\sigma_j=\{m_{j-1}+1,
\ldots, m_j\}$, then  we obtain 
$$
(1-t_{n_{j-1}})|x_{n_{j-1}}| + \sum_{ n_{j-1} < n < n_{j}}
|x_n|\ +\ t_{n_{j}}|x_{n_{j}}|=S_{\sigma_{j}}({\textstyle
\frac{1}{2}}), \quad \mbox{for\ all } j\ge 2.
$$
Now,  with $n_0=0$, choose $\beta$ such that 
\begin{displaymath}
\beta_n^{-1} = \left\{ \begin{array}{lcc}
\displaystyle t_{n_j}{\textstyle
\frac{S_{\sigma_{j}(\frac{1+\delta}{2})}}{S_{\sigma_{j}}
(\frac{1}{2})}}+(1-t_{ n_j}){\textstyle
\frac{S_{\sigma_{j+1}}(\frac{1+\delta}{2})}{S_{\sigma_{j+1}}
(\frac{1}{2})}}& \quad \mbox{if}\quad  & n=n_j,\\ \\
\frac{S_{\sigma_j}(\frac{1+\delta}{2})}{S_{\sigma_j}(\frac{1
}{2})}&\quad \mbox{if}\quad & n_{j-1}<  n < n_j .
\end{array}
\right.
\end{displaymath}

Put $c_j= {\textstyle
\frac{S_{\sigma_j}(\frac{1+\delta}2)}{S_{\sigma_j}(\frac12)}
}$, as we mentioned above  $|c_j|\ge 1$ and $c_j\to \infty$,
then $(\beta_n)_n \in B_{c_0}$. Finally, we have
 
$$
\begin{array}{ll}
 \|(\frac{x_n}{\beta_n})_n\|_{1}&= \d \sum_{n <n_1}
\beta_n^{-1}|x_n| + \beta_{n_1}^{-1}|x_{n_1}|  +  \sum_{j\ge
2} \sum_{n_{j-1}< n < n_j} \beta_n^{-1}|x_n|  + \sum_{j\ge
2} \beta_{n_j}^{-1}|x_{n_j}|\\ 
                                          &=\displaystyle
c_1\sum_{n <n_1} |x_n|+t_{n_1}c_1|x_{n_1}| 
+(1-t_{n_1})c_2|x_{n_1}| \\
& \qquad +\  \   \d\sum_{j\ge 2} \sum_{n_{j-1}< n < n_j} c_j
 |x_n| + \sum_{j\ge 2} t_{n_j} c_j |x_{n_j}| + (1-t_{n_j})
c_{j+1} |x_{n_j}| \\
 &=\displaystyle c_1 \Big(\sum_{n <n_1}
|x_n|+t_{n_1}|x_{n_1}|\Big)
 + \sum_{j\ge2}c_j \Big[(1-t_{n_{j-1}})|x_{n_{j-1}}|
+\sum_{{n_{j-1}} <n<n_j}|x_n|+t_{n_{j}}|x_{n_{j}}|\Big]\\
&=\sum_{j=1}^{\infty}
S_{\sigma_j}(\frac{1+\delta}{2})=\frac{1+\delta}{1-\delta}
<1+\ep.
\end{array}
$$
Thus, the lemma is proved.
\end{proof}
\smallskip

\begin{proposition} Let $E$ and $F$ be Banach spaces, $1\leq
p < \infty$, and $P \in \mathcal P(^mE;F)$.  The following statements are 
equivalent.
\begin{enumerate}
\item[(i)] $P \in \mathcal P_{K_p}(^mE;F)$.

\item[(ii)] There exist a sequence $(y_n)_n \in \ell_p(F)$,
a polynomial $Q \in \mathcal P(^mE;{\ell_{q}}/{\ker
\theta_y})$ and an operator $T\in \mathcal
K_p({\ell_{q}}/{\ker \theta_y};F)$ such that $P=T\circ Q$.
In this case 
$$
\kappa_p(P)=\inf\{\|Q\| \kappa_p(T)\},
$$
where the infimum is taken over all the possible
factorizations as above.

\item[(iii)] There exist a closed subspace $M \subset
\ell_1$, a polynomial $Q \in \mathcal P(^mE;{\ell_{q}}/{\ker
\theta_y})$ and operators $R \in \mathcal \mathcal
K_p({\ell_{q}}/{\ker \theta_y};{\ell_1}/{M})$ and  $S\in
\mathcal K({\ell_1}/{M};F)$ such that $P=S R  Q$. In this
case 
$$
\kappa_p(P)=\inf\{\|Q\| \kappa_p(R)\|S\| \},
$$
where the infimum is taken over all the possible
factorizations as above.
\end{enumerate}
\end{proposition}

\begin{proof} 
 It is clear that either (ii) or (iii) implies (i).

Now, assume that $P$ is a $p$-compact polynomial. By
Lemma~\ref{pcomp-lin}, we have that $P=L_P\Lambda$, where
$L_P$ is a $p$-compact operator and $\Lambda$ is a norm one
polynomial. Applying \cite[Theorem 3.2]{SiKa} to $L_p$ we
have that $L_P=\widetilde \theta_y (L_P)_y$ where 
$y=(y_n)_n\in \ell_p(F)$ is such that
$L_P(B_{\bigotimes^m_{\pi_s}E})\subseteq \pco\{y_n\}$,
with $\kappa_p(\widetilde \theta_y)=\|y\|_{\ell_p(F)}$ and
$\|(L_P)_y\| \le 1$. 

Then, $P= T Q$ for $T=\widetilde \theta_y$ and $Q=
(L_P)_y\Lambda$, see the diagram below. Also, we have that whenever $P=TQ$
as in (ii),  $\kappa_p(P)\le \kappa_p(T)\|Q\|$.
Again by Lemma~\ref{pcomp-lin}, $\kappa_p(L_P)=\kappa_p(P)$
and, since $\|\Lambda\|=1$, the infimum is attained.  

To show that (i) implies (iii), we again consider the
$p$-compact operator $L_P$ and the factorization
$L_P=\widetilde   \theta_y (L_P)_y$.  Now we modify the proof in
\cite[Theorem 3.1]{CK} to obtain the infimum equality. Fix $\varepsilon >0$ and suppose $y=(y_n)_n$ is
such that  $\|(y_n)_n\|_{\ell_p(F)}\le \kappa_p(L_P) +
\varepsilon$. Using Lemma~\ref{lemma-tech}, choose
$(\beta_n)_n \in B_{c_0}$ such that
$(\frac{y_n}{\beta_n})_n$ belongs to $\ell_p(F)$ with
$\|(\frac{y_n}{\beta_n})_n\|_{\ell_p}\leq
\|(y_n)_n\|_{\ell_p(F)} + \varepsilon$. 

For simplicity, we suppose that $y_n\ne 0$ , for all $n$.
Now, define the operator $s\colon\ell_1 \rightarrow F$ as
$s((\gamma_n)_n)=\sum_n \gamma_n y_n \frac{\beta_n}{\|y_n\|}$
which satisfies $\|s\|\le \|\beta\|_\infty\le 1$, and consider the closed subspace $M=\ker
s\subset \ell_1$.  Then, we may
set $R\colon {\ell_q}/{\ker \theta_y}\rightarrow
{\ell_1}/{M}$ the linear operator such that
$R([(\alpha_n)_n])=[(\alpha_n \frac{\|y_n\|}{\beta_n})_n]$
and, with $S$ the quotient map associated to $s$ we get the
following diagram:
$$
\xymatrix{
& E \ar[r]^P \ar[d]_{\Lambda}  & F   
                 &    \\
& \bigotimes^m_{\pi_s}E \ar[ur]^{L_P} \ar[r]_{(L_P)_y}  &
\ell_q / \ker\theta_y \ar[r]_{\quad R} \ar[u]_{\widetilde
\theta_y}  & \ell_1/M \ar[ul]_S}
$$
Then,  $P=S R  Q$, with $Q= (L_P)_y\Lambda$ and $\|S\|,
\|Q\|\le 1$.  Since $R(B_{\ell_q/ \ker \theta_y}) \subset
\pco \{[\frac{\|y_n\|}{\beta_n} e_n]\}$, by the choice
of $\beta$, $R$ is $p$-compact and $\kappa_p(R) \leq
\|(\frac{\|y_n\|}{\beta_n})_n\|_p \leq
\|(y_n)_n\|_{\ell_p(F)} + \varepsilon$.

Finally,
  $$
  \kappa_p(P)\leq  \|S\| \kappa_p(R) \|Q\| \leq
\|(y_n)_n\|_{\ell_p(F)} + \varepsilon\leq \kappa_p (L_P) +
2\varepsilon.
  $$
By Lemma~\ref{pcomp-lin}, the proof is complete.
\end{proof}

It was shown in \cite{DPS_adj}, that an operator $T\colon
E\to F$ is $p$-compact if and only if its bitranspose
$T''\colon E''\to F''$ is $p$-compact with $\kappa_p(T'')\le
\kappa_p(T)$.  In \cite{GaLaTur}, it is shown that, in fact,
$\kappa_p(T'')=\kappa_p(T)$ regardless $T''$ is considered
as an operator on $F''$ or, thanks to the Gantmacher
theorem, as an operator on $F$.  This result, allows us to
show that the Aron-Berner extension is a
$\kappa_p$-isometric extension which preserves the ideal of
$p$-compact homogeneous polynomials. Recall that $\overline
P$ denotes the Aron-Berner extension of $P$.
\smallskip

\begin{proposition} Let $E$ and $F$ be Banach spaces, 
$1\leq p <\infty$, and  $P \in \mathcal P(^mE;F)$. 
Then $P$ is $p$-compact if and only if $\overline P$ is
$p$-compact. Moreover, $\kappa_p(P)=\kappa_p(\overline P)$.
\end{proposition}

\begin{proof} Clearly, $P$ is $p$-compact whenever
$\overline P$ is and also $\kappa_p(P)\leq
\kappa_p(\overline P)$. Now, suppose that $P$ is
$p$-compact. By Lemma~\ref{pcomp-lin}, we can factorize $P$
via its linearization as $P=L_P\Lambda$, with $\|\Lambda\|=1$
and $L_P$ a $p$-compact operator.  Since $\overline P =L_P''
\overline \Lambda$, by \cite{DPS_adj},  $\overline P$ is
$p$-compact with  $\kappa_p(\overline P)\leq
\kappa_p(L_P'')$. Now, applying \cite{GaLaTur} and
Lemma~\ref{pcomp-lin},
$\kappa_p(L_P'')=\kappa_p(L_P)=\kappa_p(P)$, which gives the
reverse inequality. 
\end{proof}

We finish this section by relating the transpose of $p$-compact polynomials with quasi $p$-nuclear operators.  
Given an homogeneous
polynomial $P$ its adjoint is defined as the linear operator
$P'\colon F'\to \mathcal P(^m E)$ given by $P'(y')=y'\circ
P$. In \cite{GaLaTur}, it is shown that the transpose of a $p$-compact 
linear operator satisfies the equality $\kappa_p(T)=\nu_p^Q(T')$. Since 
$P'=L_P'$, where $L_P$ is the linearization of $P$, using Lemma~\ref{pcomp-lin} we
immediately have:

\begin{corollary}\label{transpose} An homogeneous polynomial $P\in \mathcal
P(^m E;F)$ is $p$-compact if and only if its transpose
$P'\in \mathcal L(F'; \mathcal P(^m E))$ is quasi
$p$-nuclear, and $\kappa_p(P)=\nu_p^Q(P')$. 
\end{corollary}

When this manuscript was completed we learned that R. Aron and P. Rueda were also been working on $p$-compact polynomials \cite{AR1}. They obtained Lemma~\ref{pcomp-lin} and a non isometric version of the corollary above. 

\section{The $p$-compact holomorphic mappings}

In this section we undertake
a detailed study of  $p$-compact
holomorphic mappings, whose definition recovers the notion
of compact holomorphic mappings for $p=\infty$,  \cite{AMR}.
Recall that for  $E$ and $F$  Banach spaces, $1\le p\le
\infty$, a holomorphic function $f\colon E\to F$ is said to
be $p$-compact at $x_0$ if there is a neighborhood $V_{x_0}$
of $x_0$, such that $f(V_{x_0})\subset F$ is a relatively
$p$-compact set. Also, $f\in \mathcal H(E;F)$ is said to be
$p$-compact  if it is $p$-compact at $x$ for all $x\in E$. We denote by 
 $\mathcal H_{K_p}( E; F)$ the space of $p$-compact entire functions and by  $\mathcal H_{K}( E; F)$ the space of compact holomorphic mappings. 
For homogeneous polynomials, it is equivalent to be compact
($p$-compact) at some point of $E$ or to be compact
($p$-compact) at every point of the space  \cite{AMR, AS}.
The same property remains valid for compact holomorphic
mappings \cite[Proposition~3.4]{AS}. We will see that 
the situation is very different for $p$-compact holomorphic
functions, $1\le p<\infty$. Furthermore, we will show that
$p$-compact holomorphic mappings, $1\le p<\infty$,  behave
more like nuclear than compact holomorphic functions.

Having in mind that $(\mathcal P_{K_p}(^m E; F), \kappa_p)$
is a polynomial Banach ideal with
$\kappa_p(P)=\m_p(P(B_E))$, and that all polynomials in the
Taylor series expansion of a $p$-compact holomorphic
function at $x_0$ are $p$-compact \cite[Proposition
3.5]{AMR}, we propose
to connect the concepts of $p$-compact holomorphic mappings and the size of 
$p$-compact sets measured by $\m_p$.  We start with a simple but useful
lemma.

\begin{lemma}\label{summing p-comp}
Let $E$ be a Banach space and consider $K_1, K_2,\ldots$ a
sequence of relatively $p$-compact  sets in $E$, $1\leq p <
\infty$. If $\sum_{j=1}^{\infty} \m_p(K_j) < \infty$, then
the series $\sum_{j=1}^{\infty} x_j$ is  absolutely
convergent for any choice of $x_j \in K_j$ and the set
$K=\{\sum_{j=1}^{\infty} x_j \colon x_j \in K_j\}$ is
relatively $p$-compact with $\m_p(K)\leq \sum_{j=1}^{\infty}
\m_p(K_j) < \infty$.
\end{lemma}

\begin{proof} Note that $K$ is  well defined since for
$x_j\in K_j$, 
$\|x_j\| \le \m_p(K_j)$,  for all $j$ and
$\sum_{j=1}^{\infty} \m_p(K_j) < \infty$.

First, suppose that $p>1$ and fix  $\varepsilon >0$.
For each $j \in \N$, we may assume that $K_j$ is nonempty
and we may choose $(x_n^j)_{n}\in\ell_p(E)$ such that
$K_j \subset \pco\{x_n^j\colon {n\in \mathbb N}\}$ with $\|(x_n^j)_n\|_p \leq
\m_p(K_j) (1+\dfrac \varepsilon{2^j}\m_p(K_j)^{-1})^{1/p}$.
Now, take $\lambda_j= \m_p(K_j)^{-1/q}$, where $\frac 1p +
\frac 1q=1$ and define the sequence $(z_k)_k\subset E$ such
that each term is of the form $\lambda_j x^j_n$, following
the standard order:
$$
\xymatrix{
\lambda_1 x^1_1  \ar[d]   &\lambda_1 x^1_2 \ar[r]  
&\lambda_1 x^1_3 \ar[dl] \quad \ldots \\
\lambda_2 x^2_1  \ar[ur]  &\lambda_2 x^2_2 \ar[dl] 
&\lambda_2 x^2_3 \quad \ldots \\
\lambda_3 x^3_1  \ar[d]   &\lambda_3 x^3_2 \ar[ur] 
&\lambda_3 x^3_3  \quad \ldots\\
       \ar[ur]  &           &   \\}
$$
Then 
$$
\begin{array}{r l}
\displaystyle \sum_{k=1}^{\infty} \|z_k\|^p =&\displaystyle
\sum_{j=1}^{\infty}\sum_{n=1}^{\infty} \lambda_j^p\|
x^j_n\|^p \\
                              =& \displaystyle
\sum_{j=1}^{\infty}\m_p(K_j)^{-p/q}
\|(x^j_n)_n\|^p_{\ell_p(E)} \\
                          \leq & \displaystyle
\sum_{j=1}^{\infty} \m_p(K_j)^{-p/q} \m_p(K_j)^p (1+\frac
\varepsilon{2^j}\m_p(K_j)^{-1}) \\
                             = & \displaystyle
\sum_{j=1}^{\infty} \m_p(K_j) + \varepsilon. \\
\end{array}
$$                             
Hence, $(z_k)_k$ belongs to $\ell_p(E)$ and
$\|(z_k)_k\|_{\ell_p(E)} \leq (\sum_{j=1}^{\infty} \m_p(K_j)
+ \varepsilon)^{1/p}$.

Now, if $K= \{\sum_{j=1}^{\infty} x_j \colon x_j \in
K_j\}$, 
we claim that $K\subset (\sum_{j=1}^{\infty}
\m_p(K_j))^{1/q}\pco \{z_k\}$.
Indeed, if $x \in K$, then $x=\sum_{j=1}^{\infty} x_j$ with
$x_j \in K_j$. Fix $j \in \N$, there  exists
$(\alpha_n^j)_{n} \in B_{\ell_q}$ such that
$x_j=\sum_{n=1}^{\infty} \alpha_n^j x_n^j$. Then,
$x=\sum_{j=1}^{\infty}\sum_{n=1}^{\infty} \alpha_n^j x_n^j$
and the series converges absolutely as the partial sums of
$|\alpha_n^j|\|x_n^j\|$ are clearly convergent with  the order
given above.  We may write  $x=\sum
\m_p(K_j)^{1/q} \alpha_n^j \lambda_j x_n^j $ with

$$
\sum |\m_p(K_j)^{1/q} \alpha_n^j
|^q=\sum_{j=1}^{\infty}\sum_{n=1}^{\infty} |\alpha_n^j|^{q}
\m_p(K_j)\leq \sum_{j=1}^{\infty}\m_p(K_j).
$$
Then $K\subset (\sum_{j=1}^{\infty} \m_p(K_j))^{1/q}\pco
\{z_k\}$. It follows that $K$ is $p$-compact and 
$$
\m_p(K) \leq (\sum_{j=1}^{\infty} \m_p(K_j))^{1/q}
\|(z_k)_k\|_{\ell_p(E)} \leq (\sum_{j=1}^{\infty}
\m_p(K_j))^{1/q} (\sum_{j=1}^{\infty} \m_p(K_j) +
\varepsilon)^{1/p}.
$$
Letting $\varepsilon \rightarrow 0$, we conclude that
$\m_p(K) \leq \sum_{j=1}^{\infty} \m_p(K_j)$.

With the usual modifications, the case $p=1$  follows from
the above construction considering $\lambda_j=1$, for all
$j$.
\end{proof}

Aron, Maestre and Rueda \cite[Proposition 3.5]{AMR} prove that if $f$ is a $p$-compact
holomorphic mapping at some $x_0\in E$,  every homogeneous
polynomial of the Taylor series expansion of $f$ at $x_0$ is
$p$-compact. At the light of the
existent characterization for compact holomorphic mappings \cite{AS},
 they also wonder if the converse is true \cite[Problem 5.2]{AMR}. To
tackle this question  we need to define the $p$-compact
radius of convergence of a function $f$ at $x_0\in E$. 

\begin{definition}\label{def-p-radius} Let $E$ and $F$ be Banach spaces, $f \in
\mathcal H(E;F)$ and $x_0 \in E$. If $\sum_{m=0}^\infty
P_mf(x_0)$ is the Taylor series expansion of $f$ at $x_0$,
we say that $$r_p(f,x_0)=1/\limsup
\kappa_p(P_mf(x_0))^{1/m}$$ is the radius of $p$-compact
convergence of $f$ at $x_0$, for $1\leq p <\infty$.

As usual, we understand that whenever $\limsup
\kappa_p(P_mf(x_0))^{1/m}=0$, the radius of $p$-compact
convergence is infinite. Also, if $P_mf(x_0)$ fails to be $p$-compact for some $m$, $f$ fails to be $p$-compact and  $r_p(f,x_0)=0$.
\end{definition}

The following lemma is obtained by a slight modification of the 
generalized Cauchy formula given in the proof of
\cite[Proposition~3.5]{AMR}, which asserts that if $f \in
\mathcal H(E;F)$ and $x_0 \in E$, fixed  $\varepsilon >0$ we
have that $P_mf(x_0)(B_{\varepsilon}(0)) \subset
\overline{\co} \{ f(B_{\varepsilon}(x_0))\}$, where $B_\varepsilon(x_0)$ stands for the open ball of center $x_0$ and radius $\varepsilon$. We state the
result as it will be used in Section~4, also we are
interested in measuring the $\m_p$-size of  $P_mf(x_0)(V)$
in terms of the $\m_p$-size of $f(x_0+V)$ for certain 
absolutely convex  open sets $V\subset E$.

\begin{lemma}\label{AMR-Cauchy}
Let $E$ and $F$ be Banach spaces, let $x_0 \in E$ and let
$V\subset E$ be an absolutely convex  open set. Let
$f \in \mathcal H(E;F)$ whose Taylor series expansion at
$x_0$ is given by $\sum_{m=0}^\infty P_mf(x_0)$. Then 
\begin{enumerate}
\item[(a)] $P_mf(x_0)(V)\subset \overline{\co} \{
f(x_0+V)\}$, for all $m$.
\item[(b)]  If $f(x_0+V)$ is relatively $p$-compact then  $\m_p(P_mf(x_0)(V))\leq \m_p(f(x_0+V))$, for all $m$.
\end{enumerate}
\end{lemma}

Now we are ready to give a characterization of $p$-compact
functions in terms of the polynomials in its Taylor series
expansion and the $p$-compact radius of convergence.  

\begin{proposition}\label{p-comp radius}
Let $E$ and $F$ be Banach spaces and let $f \in \mathcal
H(E;F)$ whose  Taylor series expansion  at $x_0$ is given by
$\sum_{m=0}^\infty P_mf(x_0)$. For $1\leq p < \infty$, the
following statements are equivalent.
\begin{enumerate}
\item[(i)] $f$ is $p$-compact at $x_0$.
\item[(ii)] $P_mf(x_0) \in \mathcal P_{K_p}(^mE;F)$, for all
$m$ and $\limsup \kappa_p(P_mf(x_0))^{1/m}<\infty$.
\end{enumerate}
\end{proposition}
\begin{proof} To prove that (i) implies (ii), take
$\varepsilon >0$ such that $f(B_{\varepsilon}(x_0))$ is
relatively $p$-compact and $f(x)=\sum_{m=1}^{\infty}
P_mf(x_0)(x-x_0)$, with uniform convergence  in
$B_{\varepsilon}(x_0)$. By  
\cite[Proposition~3.5]{AMR},  $P_mf(x_0)(\varepsilon B_E) 
\subset \overline{\co} \{ f(B_{\varepsilon}(x_0))\}$ and
$P_mf(x_0)$ is $p$-compact, for all $m$. Moreover, by the
lemma above,
$$
\kappa_p(P_mf(x_0))=\m_p(P_mf(x_0)(B_E))=\frac{1}{
\varepsilon ^m}\m_p(P_mf(x_0)(\varepsilon B_{E}))\leq
\frac{1}{\varepsilon ^m} \m_p(\overline{co}\{
f(B_{\varepsilon}(x_0))\}).
$$
It follows that $\limsup \kappa_p(P_mf(x_0))^{1/m}\leq
\frac{1}{\varepsilon}$, as we wanted to prove. 

Conversely, suppose that $\limsup
\kappa_p(P_mf(x_0))^{1/m}=C>0$ and choose $0< \varepsilon 
<r_p(f,x_0)$ such that, for all $x \in
B_{\varepsilon}(x_0)$, $f(x)=\sum_{m=1}^{\infty}
P_mf(x_0)(x-x_0)$, with uniform convergence. Now we have 
\begin{equation*}\label{f_inclusion_Pmf}
f(B_{\varepsilon}(x_0)) \subset \{\sum_{m=1}^{\infty} x_m
\colon x_m \in P_mf(x_0)(\varepsilon B_{E})\}. 
\end{equation*}

By Lemma~\ref{summing p-comp}, we obtain the result if
we prove that $\sum_{m=1}^{\infty}
\m_p(P_mf(x_0)(\varepsilon B_{E}))<\infty$, which follows
from the equality
$$
\sum_{m=1}^{\infty} \m_p(P_mf(x_0)(\varepsilon B_{E}))= 
\sum_{m=1}^{\infty}\varepsilon^{m}\kappa_p(P_mf(x_0)),
$$
and the choice of $\varepsilon$. 
\end{proof}

\begin{remark}\label{m_p de funcion}{\rm 
Let $f$ be a $p$-compact holomorphic mapping at $x_0$ and
let $\sum_{m=0}^\infty P_mf(x_0)$ be its Taylor series
expansion at $x_0$. Then, if $\varepsilon < r_p(f,x_0)$,  
$$
\m_p(f(B_{\varepsilon}(x_0))\leq \displaystyle
\sum_{m=1}^{\infty} \m_p(P_mf(x_0)(\varepsilon B_{E})),
$$
where the right hand series is convergent.}
\end{remark}

The $p$-compact radius has the following natural property.

\begin{proposition}\label{convergencia_radio}
Let $E$ and $F$ be Banach spaces,  $1\leq p< \infty$,  and $f
\in \mathcal H(E;F)$.  Suppose that $f$ is $p$-compact  at
$x_0$ with positive $p$-compact radius  $r=r_p(f,x_0)$. Then
$f$ is $p$-compact for all $x \in B_{r}(x_0)$. Also, if $f$
is $p$-compact at $x_0$ with infinite $p$-compact  radius,
then $f$ is $p$-compact at $x$, for all $x\in E$. 
\end{proposition}

\begin{proof}
Without loss of generality, we can assume that $x_0=0$. For
$r=r_p(f,0)$, take $x \in E, \|x\|<r$.  By 
\cite[Proposition 1, p.26]{NACH}, there exists 
$\varepsilon >0$ such that $f(y)=\sum_{m=1}^{\infty}
P_mf(0)(y)$ converges uniformly  for all $y \in
B_{\varepsilon}(x)$. We also may assume that
$\|x\|+\varepsilon<r$. 

As in Proposition~\ref{p-comp radius}, we have that
$f(B_{\varepsilon}(x))\subset \{\sum_{m=1}^{\infty} x_m
\colon x_m \in P_mf(0)(B_{\varepsilon}(x))\}$.  Now, if we
prove that $\sum_{m=1}^{\infty}
\m_p(P_mf(0)(B_{\varepsilon}(x))) < \infty$, the result will
follow from Lemma~\ref{summing p-comp}. Indeed, 
$$
\begin{array}{r l}
\displaystyle \sum_{m=1}^{\infty}
\m_p(P_mf(0)(B_{\varepsilon}(x)))=& \displaystyle
\sum_{m=1}^{\infty}
(\|x\|+\varepsilon)^m\m_p\Big(P_mf(0)(\textstyle{\frac{1}{
\|x\| + \varepsilon}} B_{\varepsilon}(x))\Big)\\
                             \leq& \displaystyle
\sum_{m=1}^{\infty}
(\|x\|+\varepsilon)^m\m_p\big(P_mf(0)(B_E)\big)\\
                             =& \displaystyle
\sum_{m=1}^{\infty} \Big((\|x\|+\varepsilon)
\kappa_p(P_mf(0))^{1/m}\Big)^m.
\end{array}
$$
Since $(\|x\|+\varepsilon) r^{-1} <1$, the last series is convergent
and the claim is proved.
\end{proof}

We recently learnt that R. Aron and P. Rueda defined, in the context of 
 ideals of holomorphic functions \cite{AR2}, a radius of $\mathcal I$-boundedness which for $p$-compact holomorphic functions coincides with Definition~\ref{def-p-radius}.  With  the radius of $\mathcal I$-boundedness they obtained  a partial version of Proposition~\ref{p-comp radius}.

Thanks to the Josefson-Nissenzweig theorem  we  have, for
any Banach spaces $E$ and $F$, a $p$-compact holomorphic
mapping, $f\in \mathcal H_{K_p}(E; F)$, whose   $p$-compact
radius of convergence at the origin is finite. It is enough
to consider a sequence $(x'_m)_m \subset E'$ with
$\|x'_m\|=1 \ \forall m \in \N$, and $(x'_m)_m$ point-wise
convergent to 0. Then, $f(x)=\sum_{m=1}^{\infty} x'_m(x)^m$
belongs to $\mathcal H(E)$,  is $1$-compact (hence,
$p$-compact for $p>1$) and $r_p(f, 0)=1$ since $\kappa_p
((x'_m)^m)=\|x'_m\|=1$. The example can be modified to
obtain a vector valued holomorphic function with similar
properties.

There are two main questions related to $p$-compact
holomorphic functions which where stated as  Problem~5.1 and
Problem~5.2 by Aron, Maestre and Rueda \cite{AMR}. Both
arise from properties that compact  holomorphic functions
satisfy. Recall that we may consider compact sets as
$\infty$-compact sets and compact mappings as
$\infty$-compact functions, where $\kappa_\infty(P)=\|P\|$,
for any compact $m$-homogeneous polynomial $P$. Let us
consider $f\in \mathcal H(E;F)$, by \cite[Proposition
3.4]{AS} it is known that if $f$ is compact at one point,
say at the origin, then $f$ is compact at $x$ for all $x\in
E$. Also, if $\sum_{m=0}^\infty P_mf(0)$ is the Taylor
series expansion of $f$ at $0$, and for each $m$ the
homogeneous polynomial $P_mf(0)\colon E\to F$ is compact,
then $f$ is compact. With Example~\ref{Pn p-comp f no} we
show that this later result is no longer true for $1\leq p
<\infty$. Note that $\limsup \|P_m\|^{1/m} <\infty$ is
fulfilled by the Cauchy's inequalities whenever $f$ is
compact. 
Example~\ref{Pn p-comp f no} also shows that, in
Proposition~\ref{p-comp radius}, the hypothesis
$\displaystyle \limsup \kappa_p(P_mf(x_0))^{1/m} < \infty$
cannot be ruled out. For our purposes, we adapt
\cite[Example 10]{DIN}.

\begin{example}\label{Pn p-comp f no}
For every $1\leq p < \infty$, there exists a holomorphic
function $f \in \mathcal H(\ell_1;\ell_p)$ such that for all
$m \in \N$, $P_mf(0)$ are $p$-compact, but $f$ is not
$p$-compact at $0$. 

Furthermore, every polynomial $P_mf(y)$ in the Taylor series
expansion of $f$ at any $y \in \ell_1$
is $1$-compact, and therefore $p$-compact for all $1\leq p
<\infty$, but $f$ is not $p$-compact at any $y$.
\end{example}
 
\begin{proof} Consider the partition of the natural numbers
given by $\{\sigma_m\}_m$, where each $\sigma_m$ is a finite
set of $m!$ consecutive elements determined as follows:
$$
\sigma_1=\{1\};\quad \sigma_2=\{\ \underbrace{2,3}_{2!}\
\};\quad \sigma_3=\{\ \underbrace{4,5,6,7,8,9}_{3!}\
\};\quad \sigma_4=\{\underbrace{\ldots}_{4!}\};\ \cdots
$$

Let $(e_j)_j$ be the canonical basis of $\ell_p$ and denote by
$(e'_j)_j$ the sequence of coordinate functionals on
$\ell_1$. Fixed $m\geq 1$,  consider the polynomial $P_m \in
\mathcal P(^m\ell_1;\ell_p)$, defined by
$$P_m(x)=(\frac{m^{m/2}}{m!})^{1/p} \sum_{j \in \sigma_m}
e'_j(x)^m e_j.$$ 
Then
$$
\|P_m\|= (\frac{m^{m/2}}{m!})^{1/p} \sup_{x \in
B_{\ell_1}}\|\sum_{j \in \sigma_m} e'_j(x)^m
e_j\|_{\ell_p}\le (\frac{m^{m/2}}{m!})^{1/p} \sup_{x \in
B_{\ell_1}} \|x\|_1^{1/p} = (\frac{m^{m/2}}{m!})^{1/p}.$$ 

First, note that $P_m$ is $p$-compact since it is of finite
rank. Now, as $\limsup \|P_m\|^{1/m}\leq \lim
(\frac{m^{1/2}}{m!^{1/m}})^{1/p}=0$, we may define $f$ as
the series  $\sum_{m=1}^{\infty}P_m$, obtaining that $f \in
\mathcal H(\ell_1;\ell_p)$.

In order to show that $f$ fails to be $p$-compact at $0$, by
Proposition~\ref{p-comp radius}, it is enough to prove that
$\limsup \kappa_p(P_m)^{1/m}=\infty$.
Fix $m \in \N$ and take $(x_n)_{n} \in
\ell_p(\ell_p)$, such that $P_m(B_{\ell_1}) \subset
\pco\{x_n\}$. Each $x_n$ may be written by
$x_n=\sum_{k=1}^{\infty} x^n_k e_k$. For each $j \in
\sigma_m$, there is a sequence $(\alpha^j_n)_n \in
B_{\ell_q}$ such that 

$$
P_m(e_j)= (\frac{m^{m/2}}{m!})^{1/p}
e_j=\sum_{n=1}^{\infty}\alpha^j_nx_n =
\sum_{n=1}^{\infty}\sum_{k=1}^{\infty} \alpha^j_n x^n_k
e_k=\sum_{k=1}^{\infty}(\sum_{n=1}^{\infty} \alpha^j_n
x^n_k) e_k.
$$

Therefore, we have that
$(\frac{m^{m/2}}{m!})^{1/p}=\sum_{n=1}^{\infty} \alpha^j_n
x^n_j$, for each $j \in \sigma_m$.
Then
$$
\begin{array}{r l}
\displaystyle m^{m/2}=\sum_{j \in \sigma_m}
\big|(\frac{m^{m/2}}{m!})^{1/p}\big|^p=&\displaystyle
\sum_{j \in \sigma_m}\big|\sum_{n=1}^{\infty} \alpha^j_n
x^n_j\big|^p\\
\leq & \displaystyle \sum_{j \in
\sigma_m}(\sum_{n=1}^{\infty} |\alpha^j_n x^n_j|)^p\\
\leq & \displaystyle \sum_{j \in
\sigma_m}(\sum_{n=1}^{\infty} |\alpha^j_n|^q)^{p/q} 
\sum_{n=1}^{\infty}|x^n_j|^p\\
\leq & \displaystyle \sum_{j \in \sigma_m}
\sum_{n=1}^{\infty}|x^n_j|^p\leq
\|(x_n)_n\|^p_{\ell_p(\ell_p)}.
\end{array}
$$

We have shown that for any sequence $(x_n)_n\in
\ell_p(\ell_p)$ such that $P_m(B_{\ell_1})\subset
\pco\{x_n\}$, the inequality
$\|(x_n)_n\|_{\ell_p(\ell_p)} \geq m^{m/2p}$ holds. Then, 
$\kappa_p(P_m)\geq m^{m/2p}$ for all $m \in \N$. Hence, we
conclude that
$\limsup \kappa_p(P_m)^{1/m}=\infty$ and, by
Proposition~\ref{p-comp radius}, $f$ cannot be $p$-compact
at $0$, which proves the first statement of the example. 

To show the second assertion, take any nonzero element $y
\in \ell_1$ and fix $m_0 \in \N$. For all $x \in B_{\ell_1}$,
$$
\begin{array}{r l}
\displaystyle P_{m_0}f(y)(x)=& \displaystyle
\sum_{m=m_0}^{\infty} \genfrac{(}{)}{0pt}{}{m}{m_0}
\overset\vee P_m(y^{m-m_0},x^{m_0})\\
                           =&\displaystyle 
\sum_{m=m_0}^{\infty} \genfrac{(}{)}{0pt}{}{m}{m_0}
\left(\frac{m^{m/2}}{m!}\right)^{1/p} \sum_{j \in
\sigma_m}e'_j(y)^{m-m_0}e'_j(x)^{m_0}e_j.
\end{array}
$$

We claim that the sequence
$\left(\genfrac{(}{)}{0pt}{}{m}{m_0}
\left(\frac{m^{m/2}}{m!}\right)^{1/p} e'_j(y)^{m-m_0}
e_j\right)_{\underset {m> m_0}{j \in \sigma_m}}$ belongs to $\ell_1(\ell_p)$. In fact,
$$
\sum_{m> m_0} \genfrac{(}{)}{0pt}{}{m}{m_0}
\left(\frac{m^{m/2}}{m!}\right)^{1/p}\sum_{j \in
\sigma_m}|e'_j(y)|^{m-m_0}\leq \sum_{m> m_0}
\genfrac{(}{)}{0pt}{}{m}{m_0}
\left(\frac{m^{m/2}}{m!}\right)^{1/p}\|y\|_1^{m-m_0}<\infty.
                   $$

Then, since $(e'_j(x)^m)_{\underset {m\geq m_0}{j \in
\sigma_m}}$ belongs to $B_{c_0}$, the set
$P_{m_0}f(y)(B_{\ell_1})$ is included in the $1$-convex hull
of
$\left\{\genfrac{(}{)}{0pt}{}{m}{m_0}
\left(\frac{m^{m/2}}{m!}\right)^{1/p} e'_j(y)^{m-m_0} e_j
\colon m\geq m_0,\ j \in \sigma_m \right\}$, which proves
that $P_{m_0}f(y)$ is 1-compact and, therefore, $p$-compact
for every $1\le p$, for any $m_0$. 

To show that $f$ is not $p$-compact at $y$, note that fixed
$m$, it is enough to choose $j \in \sigma_{m}$, to obtain
that
$P_{m}f(y)(e_{j})=\left(\frac{m^{m/2}}{m!}\right)^{1/p}e_{j}
$. Now, we can proceed as in the first part of the example
to show that $\limsup \kappa_p(P_mf(y))^{1/m}=\infty$. And,
again by Proposition~\ref{p-comp radius}, we have that $f$
cannot be $p$-compact at $y$.
\end{proof}

The following example gives a negative answer to 
\cite[Problem~5.1]{AMR}. We show an entire function which is
$p$-compact at $0$, but this property does not extend beyond
the ball $B_{r_p(f,0)}(0)$. Example~\ref{f p-com en 0 no
p-comp} proves, in addition, that  
Proposition~\ref{convergencia_radio} cannot be improved.  We
base our construction in \cite[Example 11]{DIN}.

\begin{example}\label{f p-com en 0 no p-comp}
For every $1\leq p < \infty$, there exists a holomorphic
function $f \in \mathcal H(\ell_1;\ell_p)$ such that $f$ is
$p$-compact at $0$, with $\limsup
 \kappa_p(P_mf(0))^{1/m}=1$, but $f$ is not $p$-compact at
$e_1$.
\end{example}

\begin{proof} Consider  $\{\sigma_m\}_m$, the  partition of
the natural numbers, as in Example~\ref{Pn p-comp f no}. Let
$(e_j)_j$ be the canonical basis of $\ell_p$ and denote
$(e'_j)_j$ the sequence of coordinate functionals on
$\ell_1$. 

Fixed $m\geq 2$, define $P_m \in \mathcal
P(^m\ell_1;\ell_p)$, the $m$-homogeneous polynomial
$$
P_m(x)=(\frac{1}{m!})^{1/p}e'_1(x)^{m-2}\sum_{j \in
\sigma_m}e'_j(x)^2e_j.
$$
Then
$$
\begin{array}{r l}
\displaystyle \|P_m\|=&\displaystyle (\frac{1}{m!})^{1/p}
\sup_{x \in B_{\ell_1}}(\sum_{j \in
\sigma_m}|e'_1(x)^{m-2}e'_j(x)^2|^p)^{1/p}\\
       \leq& \displaystyle (\frac{1}{m!})^{1/p}\sup_{x \in
B_{\ell_1}} (\sum_{j \in \sigma_m}|e'_j(x)|^{2p})^{1/p}\leq
(\frac{1}{m!})^{1/p}.\\
\end{array}
$$

Since $\lim \|P_m\|^{1/m} \leq \lim
(\frac{1}{m!})^{1/pm}=0$, we may define $f$ as
$f(x)=\sum_{m\geq 2} P_m(x)$, which belongs to $\mathcal
H(\ell_1; \ell_p)$ and $\sum_{m\geq 2} P_m$ is its Taylor
series expansion at $0$.

Note that each $P_m$ is $p$-compact, as it is of finite
rank, for all $m\geq2$. Moreover, when computing $\|P_m\|$,
we showed that
$\alpha(x)=\big(e'_1(x)^{m-2}e'_j(x)^2\big)_j\in B_{\ell_q}$
for all $x\in B_{\ell_1}$. Then $P_m(B_{\ell_1})\subset
(\frac{1}{m!})^{1/p} \pco \{e_j\colon j\in \sigma_m\}$ and
since $\|(e_j)_{j\in \sigma_m}\|_{\ell_p(\ell_p)}=(\sum_{j
\in \sigma_m} 1)^{1/p}=(m!)^{1/p}$, we have that
$\kappa_p(P_m)\leq  (\frac{1}{m!})^{1/p}(m!)^{1/p}=1$. Then,
$\limsup\kappa_p(P_m)^{1/m}\leq 1$ and, by
Proposition~\ref{p-comp radius}, we have that  $f$ is
$p$-compact at $0$.

To show that $r_p(f,0)=1$, fix $m\ge 2$ and $\varepsilon >0$.
Take $x_j \in B_{\ell_1}$ such that $e'_1(x_j)=1-\varepsilon$,
$e'_j(x_j)=\varepsilon$ and $e'_k(x_j)=0$ for $j \in \sigma_m$
and $k\neq j$. 

Now, fix any sequence $(y_n)_n \in \ell_p(\ell_p)$ such that
$P_m(B_{\ell_1}) \subset \pco\{y_n\}$ and write
$y_n=\sum_{k=1}^{\infty}y_k^ne_k$.

Then, for each $j \in \sigma_m$ there exists
$(\alpha_n^j)_{n} \in B_{\ell_q}$ so that 
$$
P_m(x_j)=(\frac{1}{m!})^{1/p}(1-\varepsilon)^{m-2}\varepsilon^2
e_j=\sum_{n=1}^{\infty}\alpha_n^jy_n.
$$ 

Thus, for each $j \in \sigma_m$, the equality 
$(\frac{1}{m!})^{1/p}(1-\varepsilon)^{m-2}\varepsilon^2=\sum_{n=1}
^{\infty} \alpha_n^jy^n_j$ holds. 

In consequence
$$
\begin{array}{rl}
((1-\varepsilon)^{m-2}\varepsilon^2)^p&\displaystyle=\sum_{j \in
\sigma_m} \frac{1}{m!}((1-\varepsilon)^{m-2}\varepsilon^2)^p\\
                                &\displaystyle=\sum_{j \in
\sigma_m}
|(\frac{1}{m!})^{1/p}(1-\varepsilon)^{m-2}\varepsilon^2|^p\\
                                &\displaystyle=\sum_{j \in
\sigma_m} |\sum_{n=1}^{\infty}\alpha_n^jy^n_j|^p\\
                                &\displaystyle\leq \sum_{j
\in \sigma_m} (\sum_{n=1}^{\infty}|\alpha_n^jy^n_j|)^p\\
                                &\displaystyle\leq \sum_{j
\in \sigma_m} \sum_{n=1}^{\infty} |y^n_j|^p \leq
\|(y_n)_n\|^p_{\ell_p(\ell_p)}.
\end{array}
$$
Finally, we get that $\kappa_p(P_m)\geq
(1-\varepsilon)^{m-2}\varepsilon^2$ which implies that
$\limsup \kappa_p(P_m)^{1/m}\geq1-\varepsilon$. Since
$\varepsilon>0$ was arbitrary, we obtain that $r_p(f,0)=1$.

Now, we want to prove that $f$ is not $p$-compact at $e_1$.
By Proposition~\ref{p-comp radius}, it is enough to 
show, for instance, that the $2$-homogeneous polynomial
$P_2f(e_1)\colon \ell_1\to\ell_p$ is not $p$-compact. We
have 
\begin{equation}\label{ecu}
P_2f(e_1)(x)=\sum_{m=2}^{\infty} 
\genfrac{(}{)}{0pt}{}{m}{2} \overset\vee
{P_m}(e_1^{m-2},x^2) 
\end{equation}
where $\overset\vee{P_m}$ is the symmetric $m$-linear
mapping associated to $P_m$. 

By the definition of $P_m$ we easily obtain a multilinear
mapping $A_m \in\mathcal L(^m\ell_1;\ell_p)$ satisfying 
$P_m(x)=A_m(x,\ldots,x)$, defined by
$$
A_m(x_1,\ldots,x_m)=(\frac{1}{m!})^{1/p}e'_1(x_1)\cdots
e'_1(x_{m-2})\sum_{j \in \sigma_m}e'_j(x_{m-1})e'_j(x_m)e_j.
$$
Let $\mathcal S_m$ be the symmetric group on
$\{1,\ldots,m\}$ and denote, for each $\xi\in\mathcal S_m$,
the multilinear mapping $A_m^{\xi}$ given by
$A_m^{\xi}(x_1,\ldots, x_m)=A_m(x_{\xi(1)},\ldots,
x_{\xi(m)})$. Then we have
$$
\overset \vee{P_m}(e_1^{m-2},x^2)=\frac{1}{m!}\sum_{\xi \in
S_m} A_m^{\xi}(e_1^{m-2},x^2).
$$
Since
$A_m(x_1,\ldots,x_{m-2},e_1,x_{m-1})=A_m(x_1,\ldots,x_{m-1},
e_1)=0$, for all $x_1,\ldots,x_{m-1} \in \ell_1$, and
$A_m(e_1^{m-2},x^2)=\left(\frac{1}{m!}\right)^{1/p} \sum_{j
\in \sigma_m} e'_j(x)^2 e_j$, we obtain
\begin{equation}
\overset \vee{P_m}(e_1^{m-2},x^2)=\frac{1}{m!} 2 (m-2)!
(\frac{1}{m!})^{1/p} \sum_{j \in \sigma_m} e'_j(x)^2 e_j.
\label{ecu2}
\end{equation}
Combining \eqref{ecu} and \eqref{ecu2} we get that
$$
P_2f(e_1)(x)=\sum_{m\geq 2} (\frac{1}{m!})^{1/p}\sum_{j \in
\sigma_m} e'_j(x)^2e_j.
$$
Suppose that $P_2f(e_1)$ is $p$-compact. Hence, there exists
a sequence $(y_n)_n \in \ell_p(\ell_p)$,
$y_n=\sum_{k=1}^{\infty}y^n_k e_k$ such that
$P_2f(e_1)(B_{\ell_1}) \subset \pco\{y_n\}$. For each $j
\in \sigma_m$, there exists $(\alpha_n^j)_{n} \in B_{\ell_q}$
such that $P_2f(e_1) (e_j)=(\frac{1}{m!})^{1/p}
e_j=\sum_{n=1}^{\infty} \alpha_n^j y_n$. As in the
Example~\ref{Pn p-comp f no}, we conclude that
$(\frac{1}{m!})^{1/p}=\sum_{n=1}^{\infty} \alpha_n^j y^n_j$,
if $j \in \sigma_m$.
 
Hence
$$
\begin{array}{r l}
\displaystyle \sum_{m\geq 2} \sum_{j \in \sigma_m}
((\frac{1}{m!})^{1/p})^p=& \displaystyle \sum_{m\geq 2}
\sum_{j \in \sigma_m} |\sum_{n=1}^{\infty} \alpha_n^j
y^n_j|^p \\
\leq &  \displaystyle \sum_{m\geq 2} \sum_{j \in \sigma_m}
(\sum_{n=1}^{\infty}
|\alpha_n^j|^q)^{p/q}\sum_{n=1}^{\infty} |y_n^j|^p\\
\leq & \displaystyle \sum_{m\geq 2} \sum_{j \in
\sigma_m}\sum_{n=1}^{\infty} |y_n^j|^p\\
\leq & \displaystyle \|(y_n)_n\|_{\ell_p(\ell_p)}^p<\infty,
\end{array}
$$
which is a contradiction since $\sum_{m\geq 2} \sum_{j \in
\sigma_m} ((\frac{1}{m!})^{1/p})^p$ is not convergent.
Therefore, $f$ cannot be $p$-compact at $e_1$, and the
result is proved. 
\end{proof}
\smallskip

\section{The $p$-approximation property and $p$-compact mappings}

The concept of $p$-compact sets leads naturally to that of
$p$-approximation property. A Banach space $E$ has the
$p$-approximation property if the identity can be
uniformly approximated by finite rank operators on
$p$-compact sets. Since $p$-compact sets are compact, every
space with the approximation property has the
$p$-approximation property. Then, this property can be seen
as a way to weaken the classical  approximation property.

The $p$-approximation property has been studied in \cite{CK,
DOPS} related with $p$-compact linear operators and  in
\cite{AMR} related with non linear mappings. The relation
between the approximation property and compact holomorphic
mappings was first addressed  in \cite{AS}.  Here, we are
concern with the study of the $p$-approximation property and
its relation with $p$-compact polynomials and holomorphic
functions in the spirit of \cite{AMR} and \cite{AS}. 

We start characterizing the notion of a homogeneous polynomial $P$ being $p$-compact
in terms of different conditions of continuity satisfied by
$P'$ the transpose of $P$. The first proposition gives an
answer to \cite[Problem 5.8]{AMR} and should be compared
with \cite[Proposition 3.2]{AS}. 

Before going on, some words are needed on the topologies which we will
use. We denote by  $\mathcal P_c(^mE)$ the
space $\mathcal P(^mE)$ considered with the uniform
convergence on compact sets of $E$, if $m=1$ we simply write
$E'_c$.  When compact sets are replaced by $p$-compact sets
we use the notation $\mathcal P_{cp}(^mE)$ and $E'_{cp}$. By
the Ascoli theorem, any set $L \subset \mathcal P_c(^mE)$ is relatively
compact if and only if $\sup_{P \in L} \|P\|$ is finite.
Also, if  $L\subset \mathcal P_{cp}(^mE)$ is relatively compact we
have that $L$ is point-wise bounded and then, by the
Principle of uniform boundedness, $L$ is relatively compact
in $\mathcal P_c(^mE)$.  Now we have:

\begin{proposition}\label{P'}
Let $E$ and $F$ be Banach spaces, $1\leq p <\infty$, and $P
\in \mathcal P(^mE;F)$. The following statements are
equivalent.
\begin{enumerate}
\item[(i)] $P \in \mathcal P_{K_p}(E;F)$.
\item[(ii)] $P'\colon F'_{cp}\rightarrow \mathcal P(^mE)$ is
continuous.
\item[(iii)] $P'\colon F'_{cp}\rightarrow \mathcal P_c(^mE)$
is compact.
\item[(iv)] $P'\colon F'_{cp}\rightarrow \mathcal
P_{cq}(^mE)$ is compact for any $q$, $1\leq q<\infty$.
\item[(v)] $P'\colon F'_{cp}\rightarrow \mathcal
P_{cq}(^mE)$ is compact for some $q$, $1\leq q<\infty$. 
\end{enumerate}
\end{proposition}

\begin{proof} Suppose (i) holds, then $\overline{P(B_E)}=K$
is $p$-compact and its polar set $K^{\circ}$ is a
neighborhood in $F'_{cp}$. 
For $y' \in K^{\circ}$ we have that $\|P'(y')\|=\sup_{x \in
B_E}|y'(Px)|\leq 1$, and  $P'\colon F'_{cp}\rightarrow
\mathcal P(^mE)$ is continuous. 

Now suppose (ii) holds, then there exists a $p$-compact set
$K\subset F$ such that $P'(K^{\circ})$ is equicontinuous in
$\mathcal P(^mE)$. By the Ascoli theorem, $P'(K^{\circ})$ 
is relatively compact in $\mathcal P_c(^mE)$ and $P'\colon
F'_{cp}\rightarrow \mathcal P_c(^mE)$ is compact.
 
The continuity of the identity map $\mathcal
P_c(^mE)\hookrightarrow \mathcal P_{cq}(^mE)$ gives that
(iii) implies (iv),  for all $1\leq q < \infty$. Obviously,
(iv) implies (v). To complete the proof, suppose (v) holds.
Then, there exist an absolutely convex $p$-compact set
$K\subset F$ and a compact set $L \subset \mathcal
P_{cq}(^mE)$ such that $P'(K^{\circ})\subset L$ and
therefore, there exists $c>0$ such that $\sup_{y'\in K^\circ
}\|P'(y')\|\leq c$. Note that for any $x \in c^{-\frac 1n}
B_E$ and $y' \in K^\circ$ we have that
$|P'(y')(x)|=|y'(Px)|\leq 1$. Then $P(x)\in K$, for all $x
\in c^{-\frac 1n} B_E$ and $P$ is $p$-compact.
\end{proof}

Now, we characterize the $p$-approximation property on a
Banach space in terms of the $p$-compact homogeneous
polynomials with 
values on it. In order to do so we appeal to the notion of
the $\epsilon$-product introduced by Schwartz  \cite{Schw}.
Recall that for $E$ and $F$ two locally convex spaces,
$F\epsilon E$ is defined as the space of all linear
continuous operators from $E'_c$ to $F$, endowed with the
topology of uniform convergence on all equicontinuous sets
of $E'$. The space $F\epsilon E$ is also denoted by
$\mathcal L_\epsilon(E'_c;F)$.  In \cite[Proposition
3.3]{AS} its shown, for all Banach spaces $E$ and $F$, that $(\mathcal P(^mF),\|.\|)\epsilon
E=\mathcal L_{\epsilon}(E'_{c};(\mathcal P(^m F),\|.\|))=(\mathcal P_K(^mF;E),\|.\|)$, 
where the isomorphism is
given by the transposition $P\leftrightarrow P'$. As a
consequence, it is proved that $\mathcal P(^mF)$ has the
approximation property if and only if $\mathcal
P(^mF)\otimes E$ is $\|.\|$-dense in $\mathcal P_{K}(^mF;E)$
for all Banach spaces $E$ and all $m \in \N$. We have the
following result.

\begin{proposition}\label{polinomio epsilon}
Let $E$ and $F$ be Banach spaces. Then $(\mathcal
P_{K_p}(^mF;E),\|.\|)$ is isometrically
isomorphic to $\mathcal L_{\epsilon}(E'_{c_p};(\mathcal P(^mF),\|.\|))$.

As a consequence,  $E$ has the $p$-approximation property if
and only if $\mathcal P(^mF)\otimes E$ is $\|.\|$-dense in
$\mathcal P_{K_p}(^mF;E)$ for all Banach spaces $F$ and all
$m \in \N$.
\end{proposition}

\begin{proof} Note that  [(i) implies (ii)] of
Proposition~\ref{P'}, says that the transposition  operator
maps a $p$-compact polynomial into a linear map in $\mathcal
L_{\epsilon}(E'_{c_p};\mathcal P(^mF),\|.\|))$. Now, take
$T$ an operator in $\mathcal L_{\epsilon}(E'_{c_p};\mathcal
P(^mF),\|.\|))$. Since the identity map $\iota\colon E'_c\to
E'_{cp}$ is continuous, $T$ belongs to $\mathcal
L_{\epsilon}(E'_{c};\mathcal P(^mF),\|.\|))$. By
\cite[Proposition 3.3]{AS}, we have that $T=P'$ for some
$P\in \mathcal P_K(^mF;E)$. In particular, $P'\colon
E'_{cp}\rightarrow \mathcal P(^mF)$ is continuous and by
[(ii) implies (i)] of Proposition~\ref{P'}, $P$ is
$p$-compact.  

For the second statement, if $E$ has the $p$-approximation
property, $G\otimes E$ is dense in $\mathcal
L_{\epsilon}(E'_{c_p};G)$, for every locally convex space
$G$, \cite{GaLaTur}. In particular we may consider
$G=(\mathcal P(^mF),\|.\|)$. Conversely, with $m=1$ we have
that $F'\otimes E$ is $\|.\|$-dense in $\mathcal K_p (F;E)$
for every  Banach space $F$. Now, an application of
\cite[Theorem 2.1]{DOPS} completes the proof.
\end{proof}

At the light of \cite[Proposition 3.3]{AS},  we expected to
obtain a result of the type {\it $\mathcal P(^mE)$ has the
$p$-approximation property if and only if $\mathcal
P(^mE)\otimes F$ is $\|.\|$-dense in $\mathcal
P_{K_p}(^mE;F)$ for all Banach spaces $F$ and all $m \in
\N$.} Unfortunately, our characterization is not as direct
as we wanted and requires the following notion.

\begin{definition}
Let  $E$ be a Banach space,  $\mathcal A$ an operator ideal and  $\alpha$ a norm on $\mathcal A$ . 
We say that $E$ has the $(\mathcal A,
\alpha)$-approximation property if $F'\otimes E$ is
$\alpha$-dense in $\mathcal A(F,E)$, for all Banach spaces
$F$.
\end{definition}

The relation between an ideal $\mathcal A$ with the ideal of
those operators whose transpose belongs to $\mathcal A$
leads us to work with the ideal of quasi $p$-nuclear
operators $\mathcal {QN}_p$.

\begin{proposition} Let $E$ be a Banach space and fix $m\in
\mathbb N$. Then,
\begin{enumerate}
\item[{\rm (a)}] $\mathcal P(^mE)\otimes F$ is $\|.\|$-dense
in $\mathcal P_{K_p}(^mE;F)$, for all Banach spaces $F$  if
and only if  $\mathcal P(^mE)$ has the
$(\mathcal {QN}_p,\|.\|)$-approximation property.
\item[{\rm (b)}] $\mathcal P(^mE)$ has the $p$-approximation
property if and only if $\mathcal P(^mE)\otimes F$ is
$\|.\|$-dense in $\{P \in \mathcal P(E;F) \colon L_p \in
\mathcal {QN}_p(\otimes ^{m}_{\pi_s}E;F)\}$, for all Banach spaces
$F$. 
\end{enumerate}
\end{proposition}

\begin{proof} The space $\mathcal P(^mE)$, or equivalently 
$(\otimes ^{m}_{\pi_s}E)'$ , has the
$(\mathcal{QN}_p,\|.\|)$-approximation property if and only if 
$(\otimes ^{m}_{\pi_s}E)' \otimes F$ is $\|.\|$-dense in
$\mathcal K_p(\otimes ^{m}_{\pi_s}E;F)$ for all Banach
spaces $F$, see \cite{GaLaTur}.  In virtue of
Lemma~\ref{pcomp-lin}, it is equivalent to have that
$\mathcal P(^mE)\otimes F$ is $\|.\|$-dense in $\mathcal
P_{K_p}(^mE;F)$. Then, statement (a) is proved. Note that
(a) can be reformulated saying that $\mathcal P(^mE)$ has
the $(\mathcal{QN}_p,\|.\|)$-approximation property  if and only if 
$\mathcal P(^mE)\otimes F$ is $\|.\|$-dense in $\{P \in
\mathcal P(E;F) \colon L_p \in \mathcal K_p(\otimes
^{m}_{\pi_s}E;F)\}$, for all Banach spaces $F$. 

For the proof of (b), we use  that the $p$-approximation
property corresponds to the $(\mathcal
A,\|.\|)$-approximation property for the ideal  $\mathcal A=\mathcal K_p$, 
of $p$-compact operators. The result follows proceeding as
before if the ideal $\mathcal K_p$ and its dual ideal $\mathcal{QN}_p$
are interchanged. 
\end{proof}

Now, we change our study to that of $p$-compact holomorphic
mappings.  Aron and Schottenloher described the space of
compact holomorphic functions considered with $\tau_w$, the
Nachbin topology \cite{NACH}, via the $\epsilon$-product. Namely, they
show that $(\mathcal H_K(E;F),\tau_{\omega})= \mathcal
L_{\ep}(F'_c;\mathcal H(E),\tau_{\omega})$, where the
isomorphism is given by the transposition map $f\mapsto f'$
\cite[Theorem 4.1]{AS}. The authors use this equivalence to
obtain, in presence of the approximation property, results
on density similar to that of Proposition~\ref{polinomio
epsilon}. Recall that $f'\colon F'\to \mathcal H(E)$ denotes the linear operator  given by $f'(y')=y'\circ f$.
With the next proposition we try to clarify the
relationship between $p$-compact holomorphic mappings and
the $\epsilon$-product. The result obtained gives, somehow,
a partial answer to \cite[Problem 5.6]{AMR}. 
 
\begin{proposition}\label{p-PA-eps} Let  $E$ and $F$ be
Banach spaces. Then,
\begin{enumerate}
\item[\rm(a)] $(\mathcal H_{K_p}(E;F),\tau_{\omega})$ is
topologically isomorphic to a subspace of $\mathcal
L_{\epsilon}(F'_{cp};(\mathcal H(E),\tau_{\omega}))$.
\item[\rm(b)] $\mathcal L_{\epsilon}(F'_{cp};(\mathcal
H(E),\tau_{\omega}))$ is topologically isomorphic to a
subspace of $\big\{f\in \mathcal H(E;F)\colon P_mf(x)\in
\mathcal P_{K_p}(^mE;F),\ \forall x \in E, \ \forall m \in
\N\big\}$, considered with the Nachbin topology,
$\tau_{\omega}$.
\end{enumerate}
\end{proposition}

\begin{proof}
To prove (a), fix $f$ in $\mathcal H_{K_p}(E;F)$ and 
consider $q$ any $\tau_{\omega}$-continuous seminorm  on
$\mathcal H(E)$. By  \cite[Proposition~3.47]{DIN2}, we may consider only the seminorms such that, for $g\in \mathcal H(E)$, 
$$
q(g)=\sum_{m=0}^{\infty}\|P_m g(0)\|_{K+a_mB_E},
$$ 
with $K\subset E$ an absolutely convex compact set and
$(a_m)_m$ a sequence in $c_0^+$. There exists $V\subset E$, an
open set such that $2K\subset V$ and  $f(V) \subset F$ is
$p$-compact. Fix $m_0 \in \N$ such that $2K+2a_mB_E \subset
V$, for all $m\geq m_0$. Now, choose $c>0$ such that $c(2K+2a_mB_E)
\subset 2K+2a_{m_0}B_E\subset V$, for all $m<m_0$. The polar set of $f(V)$, 
$f(V)^{\circ}$, is a neighborhood in $F'_{cp}$. By the Cauchy inequalities for entire functions, we have for all
$y' \in f(V)^{\circ}$, 

$$
\begin{array}{r l}
q(f'(y'))=&\displaystyle \sum_{m=0}^{\infty}
\|P_m\left(y'\circ f\right) (0)\|_{K+a_mB_E}\\
         =&\displaystyle  \sum_{m=0}^{\infty}\frac1{2^m}
\|P_m\left(y'\circ f\right) (0)\|_{2K+2a_mB_E}\\
         =&\displaystyle  \sum_{m<m_0} \frac1{2^m}
\|P_m\left(y'\circ f\right) (0)\|_{2K+2a_mB_E} + \sum_{m\ge
m_0} \frac1{2^m} \|P_m\left(y'\circ f\right)
(0)\|_{2K+2a_mB_E}\\
         \leq & \displaystyle  \sum_{m <m_0} \frac1{(2c)^m}
\|P_m\left(y'\circ f\right) (0)\|_{c(2K+2a_mB_E)} +
\sum_{m\ge m_0} \frac1{2^m} \|P_m\left(y'\circ f\right)
(0)\|_{2K+2a_mB_E}\\
         \leq &\displaystyle  \sum_{m <m_0} \frac1{(2c)^m}
\|y'\circ f\|_{c(2K+2a_mB_E)} + \sum_{m\ge m_0} \frac1{2^m}
\|y'\circ f\|_{2K+2a_mB_E}\\
         \leq &\displaystyle  \sum_{m<m_0} \frac1{(2c)^m}
\|y'\circ f\|_{V} + \sum_{m\ge m_0} \frac1{2^m} \|y'\circ
f\|_{V}\\
         \leq &\displaystyle  \sum_{m<m_0} \frac1{(2c)^m}  +
\sum_{m\ge m_0} \frac1{2^m} <\infty.
                 
\end{array}
$$

Then $f' \in \mathcal L (F'_{cp};(\mathcal
H(E),\tau_{\omega}))$.   Again, we use the continuity of the
identity map $\iota\colon F'_c\to F'_{cp}$ now, \cite[Theorem
4.1]{AS} implies the result.

To prove that (b) holds, take $T \in \mathcal
L(F'_{cp};(\mathcal H(E),\tau_{\omega}))$ which, in
particular, is an operator in $\mathcal L(F'_c;(\mathcal
H(E),\tau_{\omega}))$. By  \cite[Theorem 4.1]{AS},  $T=f'$
for some $f \in \mathcal H_K(E;F)$. By virtue of
Proposition~\ref{P'}, it is enough to show that
$\left(P_mf(x)\right)'\colon F'_{cp}\rightarrow( \mathcal
P(^mE),\|.\|)$ is continuous, for each $m\in \mathbb N$.
Consider $D_{x}^m\colon (\mathcal
H(E),\tau_{\omega})\rightarrow ( \mathcal P(^mE),\|.\|)$ the
continuous projection given by $D_{x}^m(g)=P_mg(x)$, for all
$g\in \mathcal H(E)$. Note that  $\left(P_mf(x)\right)'$ and
$D_{x}^m \circ f'$ coincide as linear operators.  Hence, we
obtain the result.
\end{proof}

Example~\ref{Pn p-comp f no} shows that there exists an
entire function $f\colon \ell_1\to \ell_p$, so that every
homogeneous polynomial in its Taylor series expansion at $y$
is $q$-compact for any $y\in \ell_1$, for all $1\le
q<\infty$, but $f$ fails to be $q$-compact  at $y$, for
every $y$ and every $q\le p$. However, we have the
following result. 

\begin{lemma}\label{densidad}
Let $E$ and $F$ be Banach spaces. Then, 

 $\mathcal H_{K_p}(E;F)$ is $\tau_{\omega}$-dense in $\{f\in
\mathcal H(E;F)\colon P_mf(x)\in \mathcal P_{K_p}(^mE;F),\
\forall x \in E,\ \forall m \in \N\}$.
\end{lemma}

\begin{proof}
Fix $f \in \mathcal H(E;F)$ so that $P_mf(x)\in
\mathcal P_{K_p}(^mE;F)$ for all $x \in E$ and for all $m$. 
Let $\ep>0$ and  let $q$ be  any $\tau_{\omega}$-continuous
seminorm on $\mathcal H(E;F)$  of the form 
$$
q(g)=\sum_{m=0}^{\infty}\|P_m g(0)\|_{K+a_mB_E},
$$  
with $K\subset E$ absolutely convex and compact and $(a_m)_m
\in c_0^+$. Consider $m_0 \in \N$ such that $\sum_{m\ge
m_0}\|P_mf(0)\|_{K+a_mB_E}<\ep$.  Now, let
$f_0=\sum_{m<m_0}P_mf(0)$, which is $p$-compact. Note that
$q(f-f_0)\leq \ep$ and the lemma follows.   
\end{proof}

\begin{proposition}\label{Nach-p-compact}
Let $E$ be a Banach space. Then, the following statements are
equivalent.
\begin{enumerate}
\item [\rm{(i)}] $E$ has the $p$-approximation property.
\item [\rm{(ii)}] $\mathcal H(F)\otimes E$ is
$\tau_{\omega}$-dense in $\mathcal H_{K_p}(F;E)$ for all
Banach spaces $F$.
\end{enumerate}
\end{proposition}

\begin{proof}
If $E$ has the $p$-approximation property, $E\otimes G$ is
dense in $\mathcal L_{\epsilon}(E'_{c_p};G)$ for all locally
convex space $G$  \cite{GaLaTur},  in particular if we
consider $G=(\mathcal H(F), \tau_{\omega})$. Applying
Proposition~\ref{p-PA-eps}~(a), we have the first assertion.

For the converse, put $\mathcal H_0 =\{f\in \mathcal
H(F;E)\colon P_mf(x)\in \mathcal P_{K_p}(^mF;E),\ \forall x
\in E,\ \forall m \in \N\}$. By Lemma~\ref{densidad},
$\mathcal H(F)\otimes E$ is $\tau_{\omega}$-dense in
$\mathcal H_0$. Now, take $T \in \mathcal K_p (F;E)$ and
$\ep>0$. Since $T \in \mathcal H_0$ and $q(f)=\|P_1f(0)\|$
is a $\tau_{\omega}$-continuous seminorm, there exists $g
\in \mathcal H(F)\otimes E$ such that $q(T-g) \leq \ep$. But
$q(T-g)= \|T-P_1g(0)\|$ and since $P_1g(0) \in F'\otimes E$,
we have shown that $F'\otimes E$ is $\|.\|$-dense in
$\mathcal K_p (F;E)$. By \cite[Theorem 2.1]{DOPS}, $E$ has
the $p$-approximation property.
\end{proof}

\section{Holomorphy types and topologies}
%

In this section we show that $p$-compact holomorphic
functions fit into the framework of holomorphy types. Our
notation and terminology  follow that given in \cite{DIN}.
Since, $\mathcal P_{K_p}(^mE;F)$ is a subspace of $\mathcal P(^mE;F)$ and $\mathcal P_{K_p}(^0E;F)=F$, the first two conditions in the definition of a holomorphy type are fulfilled. Therefore, we only need to  corroborate that the sequence $(\mathcal P_{K_p}(^mE;F), \kappa_p)_m$ satisfies the third condition. Indeed, this  last condition will be also fulfilled if we show 
\begin{equation}\label{holo-type}
\kappa_p(P_l(P)(a)) \le (2e)^m\kappa_p(P)\|a\|^{m-l},
\end{equation}
for every $P\in \mathcal P_{K_p}(^mE;F)$,  for all $l=1,\ldots, m$ and for all $m$, where $P_l(P)(a)$ denotes the $l$-component in the expansion of $P$ at $a$. 

A function $f\in \mathcal H(E;F)$ is said to be of holomorphic type $\kappa_p$, at $a$, if there exist $c_1, c_2>0$ such that each component of its Taylor series expansion, at $a$, is a $p$-compact polynomial satisfying that $\kappa_p(P_m f(a)) \le c_1 c_2^m$.

To give a simple proof of
the fact that  $(\mathcal P_{K_p}(^mE;F), \kappa_p)_m$ satisfy the inequalities given in \eqref{holo-type} we use the following notation. Let $P \in \mathcal P(^mE;F)$
and fix $a \in E$, we denote  by $P_{a^l}$ the
$(m-l)$-homogeneous polynomial defined as 
$$
P_{a^l}(x):~=\overset\vee{P}(a^l,x^{m-l}),$$
for all  $x\in E$ and $l<m$. Note that, for any $j<l<m$, we
have that $P_{a^l}=(P_{a^{l-j}})_{a^j}$ and that
$P_l(P)(a)=\genfrac{(}{)}{0pt}{}{m}{m-l}P_{a^{m-l}}$. We
 appeal to the description of $P_a$ given in \cite[Corollary~
1.8,~b)]{CaDiMu}: 
\begin{equation}\label{cdm}
P_a(x)=\overset\vee{P}(a,x^{m-1})=\frac{1}{m^2}\frac{1}{
(m-1)^{m-1}}\sum_{j=1}^{ m-1}P((m-1)r^j x+a).
\end{equation}
where $r \in \mathbb C$ is such that $r^m=1$ and $r^j\neq 1$
for $j<m$.

\begin{theorem}\label{tipoholo}
For any Banach spaces $E$ and $F$, the sequence $(\mathcal P_{K_p}(^mE;F), \kappa_p)_m$
is a holomorphy type from $E$ to $F$.
\end{theorem}

\begin{proof} If $ P \in \mathcal P_{K_p}(^mE;F)$ by
\cite[Proposition 3.5]{AMR} or Proposition~\ref{p-comp
radius} we have that $P_j(P)(a) \in \mathcal P_{K_p}(^jE;F)$
for all $a \in E$, for all $j\le m$. To prove the holomorphy
type structure, we will show that
$\kappa_p(P_j(P)(a))\leq 2^me^m\|a\|^{m-j}\kappa_p(P)$, for
all $j\le m$. 

Fix $a \in E$.  If we show that $\kappa_p(P_{a})\leq e
\|a\|\kappa_p(P)$ then the proof is complete using a
generalized inductive reasoning. Indeed, suppose that for
any $p$-compact homogeneous polynomial $Q$, of degree less
than $m$, the inequality $\kappa_p(Q_{a})\leq e
\|a\|\kappa_p(Q)$ holds. Then, since
$P_{a^l}=(P_{a^{l-1}})_a$ and
$P_j(P)(a)=\genfrac{(}{)}{0pt}{}{m}{m-j} P_{a^{m-j}}$, we
obtain
$$
\begin{array}{r l}
\kappa_p(P_j(P)(a))& =\genfrac{(}{)}{0pt}{}{m}{m-j}
\kappa_p(P_{a^{m-j}}) =\genfrac{(}{)}{0pt}{}{m}{m-j}
\kappa_p((P_{a^{m-j-1}})_a)\\
                    & \leq \genfrac{(}{)}{0pt}{}{m}{m-j}
\displaystyle  e \|a\|\kappa_p((P_{a^{m-j-1}}))\\
                    & \leq \genfrac{(}{)}{0pt}{}{m}{m-j}
\displaystyle  e^{m-j} \|a\|^{m-j} \kappa_p(P)\\
                    &\displaystyle \leq 2^me^m \|a\|^{m-j}
\kappa_p(P).
\end{array}
$$                    

Now, take $P \in \mathcal P_{K_p}(^mE;F)$. Then
\begin{equation}\label{1ro}
\kappa_p(P_a)=\m_p(\overset\vee{P}(a,B^{m-1}
_E))=\|a\|\m_p(\overset\vee{P}({\textstyle \frac{a}{ \|a\|}},B^{m-1}_E)).
\end{equation}

Using (\ref{cdm}) and  Lemma~\ref{summing p-comp} we have
\begin{equation}\label{2do}
\|a\|\m_p(\overset\vee{P}(\frac{a}{\|a\|},B^{m-1}_E))\leq
{\textstyle \|a\|\frac{1}{m^2}\frac{1}{(m-1)^{m-1}}}\sum_{j=1}^{m-1}
{\textstyle \m_p(P((m-1)r^jB_E+\frac{ a}{\|a\|}))}.
\end{equation}

Since $ \sup\{ \|x\|\colon\ x \in
(m-1)r^jB_E+\frac{a}{\|a\|}\}=m$, 
\begin{equation}\label{3ro}
\begin{array}{rcl}
 \|a\|\m_p(\overset\vee{P}(\frac{a}{\|a\|},B^{m-1}_E)) &\leq & \frac{\|a\|}{m^2
(m-1)^{m-1}} {\displaystyle \sum_{j=1}^{m-1}}\m_p(P((m-1)r^jB_E+\frac{a}{
\|a\|}))\\
& = & \frac{\|a\|}{m^2
(m-1)^{m-1}}{\displaystyle \sum_{j=1}^{m-1}} m^m\m_p\big(P\big(\textstyle{\frac 1m (
(m-1)r^jB_E+\frac{a}{\|a\|}})\big)\big)\\ 
&  \leq & \|a\| (\frac m{m-1})^{m-1}
\kappa_p(P)\leq e\|a\|\kappa_p(P).
\end{array}
\end{equation}

Combining (\ref{1ro}), (\ref{2do}) and (\ref{3ro}) we get that
$\kappa_p(P_{a})\leq e \|a\|\kappa_p(P)$, as we wanted to
show.
\end{proof}

As a consequence we have the following result.

\begin{corollary} Let $f$ be a function in $\mathcal H(E;F)$, then $f \in \mathcal H_{K_p}(E;F)$ if and only
if $f$ is of $\kappa_p$-holomorphy type.
\end{corollary}

\begin{proof}
It follows from  Theorem~\ref{tipoholo} and \cite[Proposition 3.5]{AMR} or
Proposition~\ref{p-comp radius}.
\end{proof}

\begin{remark}{\rm
Theorem~\ref{tipoholo} can be improved. Indeed, the same
proof of Theorem~\ref{tipoholo} shows that the sequence
$(\mathcal P_{K_p}(^mE;F))_m$ is a coherent sequence
associated to the operator ideal $\mathcal K_p(E;F)$ (see
\cite{CaDiMu} for definitions)}.
\end{remark}

Since $\mathcal H_{K_p}(E,F)$ is a holomorphy type,
following~\cite{NACH2} we have a natural topology defined on $\mathcal H_{K_p}(E,F)$ denoted
by $\tau_{\omega ,\m_p}$. This topology may be generated by different families of continuous seminorms. The original set of seminorms used to define   $\tau_{\omega ,\m_p}$ corresponds to 
the family of seminorms given below in Theorem~\ref{top-equiv}, item (c). Our aim is to 
to characterize the $\kappa_p$-approximation property of a Banach space $E$ in
an analogous way to  \cite[Theorem 4.1]{AS}. In order to do so, 
we will give different descriptions of $\tau_{\omega ,\m_p}$.  First, we need the following
result.

\begin{proposition}\label{p-compact en 0}
Let $E$ and $F$ be Banach spaces. Then, $f \in \mathcal
H_{K_p}(E;F)$ if and only if, for all $m$, $P_mf(0) \in
\mathcal P_{K_p}(^mE;F)$ and for any absolutely convex
compact set $K$, there exists $\varepsilon >0$ such that
$\sum_{n=0}^{\infty}\m_p(P_mf(0)(K+\varepsilon
B_E))<\infty$.
\end{proposition}

\begin{proof}
Take $f \in \mathcal H_{K_p}(E;F)$ and $K$ an absolutely
convex compact set. Then, $2K$ is also absolutely convex and
compact. For each $x \in 2K$, there exist $\ep_{x}>0$ such
that $f(x+\ep_xB_E)$ is $p$-compact. Now, we choose
$x_1,\ldots, x_n \in 2K$ such that $K\subset
\bigcup_{j=1}^{n} (x_j+\ep_{x_j}B_E)$ and with
$V=\bigcup_{j=1}^{n} (x_j+\ep_{x_j}B_E)$ we have that $f(V)$
is $p$-compact. Let $d=\dist(2K,\mathcal CV)>0$, where
$\mathcal CV$ denotes the complement of $V$. Let us consider
$W=2K + dB_E$, then $W$ is an absolutely convex open set and
$2K\subset W\subset V$.
Then, applying Proposition~\ref{AMR-Cauchy} we have 
$$
\sum_{n=0}^{\infty} \m_p
(P_mf(0)(K+d/2B_E)=\sum_{n=0}^{\infty} (1/2)^m \m_p
(P_mf(0)(W))\leq 2 \m_p(f(W))< \infty,
$$
which proves the first claim. 

Conversely, let $f \in \mathcal H(E;F)$ satisfy the
conditions in the proposition.  We have to show that $f$ is
$p$-compact at $x$ for any  fixed $x \in E$. Consider the
absolutely convex compact set $K$, given by  $K=\{\lambda x
\colon |\lambda|\leq 1\}$. Then,  there exists
$\varepsilon_1 > 0$ such that
$\sum_{n=0}^{\infty}\m_p(P_mf(0)(K+\varepsilon_1
B_E))<\infty$. Since $f$ is an entire function, by
\cite[Proposition 1, p.26]{NACH}, there exists
$\varepsilon_2 >0$ such that $f(y)=\sum_{m=1}^{\infty}
P_mf(0)(y)$ uniformly for $y \in B_{\varepsilon_2}(x)$. Let
$\varepsilon=\min\{\varepsilon_1; \varepsilon_2\}$, then
$f(B_{\varepsilon}(x)) \subset \{\sum_{m=0}^{\infty} x_m
\colon x_m \in P_mf(0)(B_{\varepsilon}(x))\}$. 

Also 
$$
\sum_{m=0}^{\infty} \m_p(P_mf(0)(B_{\varepsilon}(x))\leq
\sum_{m=0}^{\infty}\m_p(P_mf(0)(K+\varepsilon_1B_E))<\infty.
$$
Now, applying Lemma~\ref{summing p-comp} we obtain that $f$
is $p$-compact at $x$, and the proof is complete.
\end{proof}

The next characterization of the topology $\tau_{\omega
,\m_p}$ associated to the holomorphy type  $\mathcal H_{K_p}
(E;F)$ follows that of \cite{DIN} and \cite{NACH}. 

\begin{theorem}\label{top-equiv}
Let $E$ and $F$ be Banach spaces and consider the space
$\mathcal H_{K_p} (E;F)$. Any of the  following families of
seminorms generate the topology $\tau_{\omega ,\m_p}$.
\begin{enumerate}
\item[(a)] The seminorms $p$ satisfying that there exists a
compact set $K$ such that for every open set $V\supset K$
there exists $C_V>0$ so that 
$$
p(f)\leq C_V \m_p(f(V)) \ \ \ \forall f \in \mathcal
H_{K_p}(E;F).
$$
In this case, we say that $p$ is $\m_p$-ported by compact
sets.

\item[(b)] The seminorms $p$ satisfying that there exists an
absolutely convex compact set $K$ such that 
for every absolutely convex open set $V\supset K$ there
exists $C_V>0$ so that 
$$
p(f)\leq C_V \m_p(f(V)) \ \ \ \forall f \in \mathcal H_{K_p}
(E;F).
$$
In this case, we say that $p$ is AC-$\m_p$-ported by 
absolutely convex compact sets. 

\item[(c)] The seminorms $p$ satisfying that there exists an
absolutely convex compact set $K$ such that, for all
$\varepsilon >0$ exists $C(\varepsilon)>0$ so that 
$$ 
p(f)\leq C(\ep) \sum_{m=0}^{\infty} \ep^m \sup_{x \in K}
\kappa_p(P_mf(x)) \ \ \ \forall f \in \mathcal H_{K_p}
(E;F).
$$

\item[(d)] The seminorms $p$ satisfying that there exists an
absolutely convex compact set $K$ such that, for all
$\varepsilon >0$ exists $C(\varepsilon)>0$ so that 
$$
p(f)\leq C(\varepsilon) \sum_{m=0}^{\infty}
\m_p(P_mf(0)(K+\varepsilon B_E) \ \ \ \forall f \in \mathcal
H_{K_p} (E;F).
$$

\item[(e)] The seminorms of the form
$$
p(f)=\sum_{m=0}^{\infty} \m_p(P_mf(0)(K+a_mB_E)),
$$ 
where $K$ ranges over all the absolutely convex compact sets
and $(a_m)_m \in c_0^+$.
\end{enumerate}
\end{theorem}

\begin{proof}
First note that if $f$ is $p$-compact and $K$ is a compact
set, there exists a open set $V\supset K$ such that $f(V)$
is $p$-compact. Then, seminorms in (a) and (b) are well
defined on $\mathcal H_{K_p}(E;F)$. Also, in virtue of
Proposition~\ref{p-compact en 0}, seminorms in (d) and
(e) are well defined. Standard arguments show that
seminorms in (a) and (b) define the same topology. 

Now we show that seminorms in (b) and (c) coincide. Let
$p$ be a seminorm and let $K$ be an absolutely convex
compact set satisfying the conditions in (c).  Let $V\supset
K$ be any absolutely convex open set and take
$d=\dist(K,\mathcal C V)>0$.  By
Proposition~\ref{AMR-Cauchy}, since $K +dB_E\subset V$, we
get 
$$
\m_p(P_mf(x)(dB_E)) \leq \m_p(f(x+dB_E))\leq \m_p(f(V)), 
$$
for all  $f \in \mathcal H_{K_p}(E;F)$. Thus, 
$$
d^m \sup_{x \in K} \kappa_p(P_mf(x)) \leq \m_p(f(V)),
$$
for each $m$. Hence 

$$
p(f)\leq C(\textstyle{\frac d2}) \d\sum_{m=0}^{\infty}
(\textstyle{\frac d2})^m \sup_{x \in K}
\kappa_p(P_mf(x))\leq 2 C(\textstyle{\frac d2}) \m_p(f(V)),
$$
which shows that  $p$ is  AC-$\m_p$-ported by $K$. 

Conversely,  let $p$ be a seminorm, let $K$ be an absolutely
convex compact set satisfying the conditions in (b). Fix 
$\varepsilon >0$ and take $x_1,\ldots,x_n$ in $K$ such that
$K\subset V$ with $V=\bigcup_{j=1}^{n} B_\varepsilon(x_j)$.
As we did before, we may find  an  absolutely convex open
set $W$ so that $K\subset W \subset V$. Let $f \in \mathcal
H_{K_p}$, without loss of generality we may assume that 
$\varepsilon < r_p(f,x)$ for all $x \in K$. By
Remark~\ref{m_p de funcion}, we obtain 
$$
\m_p(f(B_\varepsilon(x_j)))\leq \sum_{m=0}^{\infty}
\varepsilon^m \kappa_p(P_mf(x_j))\leq \sum_{m=0}^{\infty}
\varepsilon^m \sup_{x \in K} \kappa_p(P_mf(x)).
$$
As $p$ is  AC-$\m_p$-ported by $K$, we have that $p(f)\leq
C_W\m_p(f(W))\leq C_W\m_p(f(V))$ and therefore 
$$
\begin{array}{r l}
p(f)&\leq C_W\d \sum_{j=1}^{n}
\m_p(f(B_\varepsilon(x_j)))\\ 
                                       &\leq C_W \d
\sum_{j=1}^{n} \sum_{m=0}^{\infty} \varepsilon^m \sup_{x \in
K} \kappa_p(P_mf(x))\\
                                       &= n C_W  \d
\sum_{m=0}^{\infty} \varepsilon^m \sup_{x \in K}
\kappa_p(P_mf(x)).
\end{array}
$$
Thus $p$ belongs to the family in (c). 
If $\varepsilon \geq r_p(f,x)$, then
$\sum_{m\ge 0} \varepsilon^m \sup_{x \in K}
\kappa_p(P_mf(x))=\infty$ and the inequality follows.

By the proof of \cite[Proposition 4]{DIN}, we have that
seminorms in (d) and (e) generate the same topology.
Finally, we show that seminorms in (d) and (b) are
equivalent. The proof of Proposition~\ref{p-compact en 0}
shows that seminorms in (d) are AC-$\m_p$-ported by
absolutely convex compact sets.

To conclude  the proof, consider a seminorm $p$ and an
absolutely convex compact set $K$ satisfying conditions in
(b). We borrow some ideas of \cite[Chapter 3]{DIN2}. For
each $m$, let $W_m$ be the absolutely convex open set
defined by $W_m = K+(\frac 12)^mB_E$. Since $p$ is
AC-$\m_p$-ported by $K$, for each $m \in \N$, there exists a
constant $C_m=C_{W_m}$ such that $p(f)\leq C_m
\m_p(f(W_m))$, every $p$-compact function $f$.

For $m=1$, there exists $n_1 \in \N$, such that for all $n>n_1$,
$C_1^{1/n}<2$. Take $V_1=2W_1$. Now, if  $n>n_1$ and $Q \in
\mathcal P_{K_p}(^nE;F)$, 
$$
p(Q)\leq C_1 \m_p(Q(W_1))=\m_p(Q(C_1^{1/n}W_1))\leq
\m_p(Q(V_1)).
$$
For $m=2$, there exists $n_2>n_1$ such that $C_2^{1/n}\leq
2$, for all $n>n_2$. Now, take $V_2=2W_2$ and, as before, 
we have for any $Q \in \mathcal P_{K_p}(^nE;F)$, with
$n>n_2$, 
$$
p(Q)\leq C_2\m_p(Q(W_2))=\m_p(Q(C_2^{1/n}W_2))\leq
\m_p(Q(V_2)).
$$ 
Repeating this procedure we obtain a sequence of absolutely
convex open sets $V_j$  satisfying
$$
\begin{array}{r l}
p(f)\leq \displaystyle
\sum_{m\ge 0} p(P_mf(0))&\displaystyle =\sum_{m< n_1}
p(P_mf(0))+\sum_{j\ge 1}\sum_{n_j \le m <n_{j+1}}
p(P_mf(0))\\
                            &\displaystyle \leq C_{V_1}
\sum_{m<n_1}\m_p(P_mf(0)(V_1))+\sum_{j\ge 1}\sum_{n_j \le m
<n_{j+1}} \m_p(P_mf(0)(V_j))\\
                            &\displaystyle \leq
C\left(\sum_{m <n_1}\m_p(P_mf(0)(V_1))+\sum_{j\ge 1}
\sum_{n_j \le m <n_{j+1}} \m_p(P_mf(0)(V_j))\right)
\end{array}
$$
where $C=\min\{1,C_{V_1}\}$ and the result follows since
$V_j =2K + (\frac 12)^{j-1}B_E$ and the seminorm $p$ is
bounded above by a seminorm of the family of the form (e). Now, the
proof is complete. 
\end{proof}

We finish this section by inspecting the
$\kappa_p$-approximation property introduced in
\cite{DPS_dens}. We 
will show that  $p$-compact homogeneous polynomials from $F$
to $E$ can be $\kappa_p$-approximated by polynomials in   $\mathcal
P(^mF)\otimes E$  whenever $E$ has the
$\kappa_p$-approximation property. We then obtain a similar
result for $p$-compact holomorphic functions. What follows
keeps the spirit of \cite[Theorem 4.1]{AS}. Recall that a
Banach space $E$ has the $\kappa_p$-approximation property
if for every Banach space $F$, $F'\otimes E$ is
$\kappa_p$-dense in $\mathcal K_p(F;E)$.

\begin{theorem}\label{Pol-kpAP}
Let $E$ be a Banach space.  The following statements are  equivalent.
\begin{enumerate}
\item[(i)] $E$ has the $\kappa_p$-approximation property.
\item[(ii)] For all $m \in \N$, $\mathcal P(^mF)\otimes E$
is $\kappa_p$-dense in $P_{K_p}(^mF,E)$, for every Banach
space $F$.
\item[(iii)] $\mathcal H(F)\otimes E$ is $\tau_{\omega,
\m_p}$-dense in $\mathcal H_{K_p}(F;E)$ for all Banach
spaces $F$.
\end{enumerate}
\end{theorem}
\begin{proof}
First, suppose that $E$ has the $\kappa_p$-approximation property and fix
$m\in \mathbb N$. Then, $(\bigotimes^m_{\pi_s}F)'\otimes E$
is $\kappa_p$-dense in $K_p(\bigotimes^m_{\pi_s}F;E)$ which,
  by Proposition~\ref{pcomp-lin}, coincides with 
$(P_{K_p}(^mF,E),\kappa_p)$,  via the isomorphism given by
$P\mapsto L_P$.   Thus, (ii) is satisfied. 

Now, assume (ii) holds.  Take $f \in \mathcal H_{K_p}(F,E)$,
$\varepsilon >0$. By Theorem~\ref{top-equiv}, we may consider a
$\tau_{\omega,\m_p}$-continuous seminorm of the form
$q(f)=\sum_{m=0}^{\infty} \m_p(P_mf(0)(K+a_mB_F))$, where
$K\subset F$ is an absolutely convex compact set and
$(a_m)_m \in c_0^+$. Let $m_0 \in \N$ be such that
$\sum_{m>m_0}\m_p(P_mf(0)(K+a_mB_F))\leq  \frac \varepsilon 2$
and let $C>0$ be such that 
$\frac 1C (K+a_mB_F) \subset B_F$, for all $m\le m_0$. Given
 $\delta>0$, to be chosen later, by hypothesis,
we may find  $Q_m \in \mathcal P(^mF)\otimes E$ such that
$\kappa_p(P_mf(0) -Q_m)\leq \delta$, for all $m\le m_0$.
Define $g=\sum_{m=0}^{m_0} Q_m$, which belongs to $ \mathcal
H(F)\otimes E$, then

$$
\begin{array}{r l}
 q(f-g)&=\displaystyle \sum_{m=0}^{m_0}
\m_p((P_mf(0)-Q_m)(K+a_mB_F))+\sum_{m> m_0}
m_p(P_mf(0)(K+a_mB_F))\\
&\displaystyle \leq \sum_{m=0}^{m_0}C^{m}
\kappa_p((P_mf(0)-Q_m)) +\frac \varepsilon 2. 
\end{array}
$$
Thus,  $q(f-g) < \varepsilon$ for a suitable choice of
$\delta$, which proves  (iii). 

Finally, suppose we have (iii). 
Take $T \in K_p(F,E)$, $\varepsilon >0$ and the seminorm on
$\mathcal H_{K_p}(F;E)$ defined by $q(f)=\kappa_p(P_1f(0))$.
Since $q$ is $\tau_{\omega, \m_p}$-continuous, by
assumption, there exist $f_1,\ldots,f_n \in \mathcal H(F)$
and $x_1,\ldots, x_n \in E$, such that
$q(T-\sum_{j=1}^{n}f_j\otimes x_j)<\varepsilon$. In other
words, $\kappa_p(T-\sum_{j=1}^{n} P_1f_j(0)\otimes
x_j)<\varepsilon$ which proves that $F'\otimes E$ is
$\kappa_p$-dense in $K_p(F,E)$.  Whence, the proof is complete. 
\end{proof}

\subsection*{Acknowledgements} We wish to express our
gratitude to Richard Aron for introducing us to the problem
during his visit to Buenos Aires in 2007 and for many useful
conversations thereafter. Also, we are grateful to Manuel
Maestre for sharing with us  the open problems in a
preliminary version of \cite{AMR}. Finally, we thank Nacho
Zalduendo for the suggestions he made to consider holomorphy
types and to Christopher Boyd for his careful reading and comments on the manuscript.

\end{document}